\documentclass[12pt,a4paper]{amsart}

\usepackage[english]{babel}
\usepackage[utf8]{inputenc}
\usepackage[T1]{fontenc}
\usepackage{latexsym}
\usepackage{amsmath,amssymb,amsthm,amsfonts}
\usepackage{mathtools}
\usepackage{relsize}
\usepackage{framed, color}
\usepackage{hyperref}
\usepackage{dsfont}
\usepackage{todonotes}
\usepackage{cancel}
\usepackage[normalem]{ulem}
\usepackage{xcolor}
\usepackage[capitalize]{cleveref}
\usepackage{geometry}

\usepackage{enumerate}
\usepackage[shortlabels]{enumitem}
\setlist[enumerate]{label=(\alph*),font=\normalshape}
\setlist[itemize]{font=\normalshape}

\allowdisplaybreaks

\theoremstyle{plain}
\newtheorem{theorem}{Theorem}[section]
\newtheorem{definition}[theorem]{Definition}
\newtheorem{proposition}[theorem]{Proposition} 
\newtheorem{lemma}[theorem]{Lemma} 
\newtheorem{corollary}[theorem]{Corollary}
\theoremstyle{definition}
\newtheorem{remark}[theorem]{Remark}

\newcommand{\C}{\mathbb{C}}
\newcommand{\R}{\mathbb{R}}

\newcommand{\Z}{\mathbb{Z}}
\newcommand{\T}{\mathbb{T}}
\newcommand{\Zodd}{\mathbb{Z}_\text{odd}}

\newcommand{\N}{\mathbb{N}}

\newcommand{\impvar}{\,\cdot\,}

\newcommand{\embeds}{\hookrightarrow}
\newcommand{\ee}{\mathrm{e}}
\newcommand{\ii}{\mathrm{i}}
\let\bd\partial

\newcommand{\abs}[1]{\left\lvert#1\right\rvert}
\newcommand{\norm}[1]{\left\lVert#1\right\rVert}
\newcommand{\seminorm}[1]{\left[#1\right]}

\newcommand{\set}[1]{\left\{#1\right\}}

\newcommand{\calC}{\mathcal{C}}

\newcommand{\der}{\,\mathrm{d}}

\DeclareMathOperator{\sign}{sign}

\DeclareMathOperator{\ft}{\mathcal{F}}

\renewcommand{\Im}{\operatorname{Im}}

\newcommand{\pdv}[3][]{%
	\if\relax\detokenize{#1}\relax%
		\frac{\partial #2}{\partial #3}%
	\else%
		\frac{\partial^{#1} #2}{\partial {#3}^{#1}}%
	\fi%
}
\newcommand{\dv}[3][]{%
	\if\relax\detokenize{#1}\relax%
		\frac{\mathrm{d} #2}{\mathrm{d} #3}%
	\else%
		\frac{\mathrm{d}^{#1} #2}{{\mathrm{d} #3}^{#1}}%
	\fi%
}

\newcommand{\tri}{\Delta}
\newcommand{\lefttri}{\Delta_-}
\newcommand{\righttri}{\Delta_+}

\let\originalleft\left
\let\originalright\right
\renewcommand{\left}{\mathopen{}\mathclose\bgroup\originalleft}
\renewcommand{\right}{\aftergroup\egroup\originalright}

\makeatletter
\newcommand{\clonelabel}[2]{\@bsphack%
  \expandafter\ifx\csname r@#2\endcsname\relax%
  \else\protected@write\@auxout{}{\string\newlabel{#1}%
    {\csname r@#2\endcsname}}%
  \fi%
  \expandafter\ifx\csname r@#2@cref\endcsname\relax%
  \else\protected@write\@auxout{}{\string\newlabel{#1@cref}%
    {\csname r@#2@cref\endcsname}}%
  \fi%
  \@esphack}
\makeatother

\newcommand{\spacext}[1]{C^1_{(x,t)}(#1)}
\newcommand{\spacex}[1]{C^1_x(#1)}

\usepackage{refcheck}	
\makeatletter
\newcommand{\refcheckize}[1]{%
  \expandafter\let\csname @@\string#1\endcsname#1%
  \expandafter\DeclareRobustCommand\csname relax\string#1\endcsname[1]{%
    \csname @@\string#1\endcsname{##1}\wrtusdrf{##1}}%
  \expandafter\let\expandafter#1\csname relax\string#1\endcsname
}
\makeatother
\refcheckize{\cref}
\refcheckize{\Cref}

\norefnames
\nocitenames

\usepackage{csquotes}
\usepackage[backend=biber]{biblatex} 
\bibliography{bibliography}

\begin{document}
	\title[Wellposedness for a (1+1)-dimensional wave equation]{Wellposedness for a (1+1)-dimensional wave equation with quasilinear boundary condition}
	
	\author{Sebastian Ohrem}
	\address{Institute for Analysis, Karlsruhe Institute of Technology (KIT), D-76128 Karlsruhe, Germany}\email{sebastian.ohrem@kit.edu}
	
	\author{Wolfgang Reichel}
	\address{Institute for Analysis, Karlsruhe Institute of Technology (KIT), D-76128 Karlsruhe, Germany}\email{wolfgang.reichel@kit.edu}
	
	\author{Roland Schnaubelt}
	\address{Institute for Analysis, Karlsruhe Institute of Technology (KIT), D-76128 Karlsruhe, Germany}\email{roland.schbaubelt@kit.edu}

	\date{\today} 
	
	\subjclass[2000]{Primary: 35L05, 35L20; Secondary: 35Q60, 35Q61}


    \keywords{Maxwell equations, wave equation, nonlinear boundary condition, wellposedness}
	

\begin{abstract}
	We consider the linear wave equation $V(x) u_{tt}(x, t) - u_{xx}(x, t) = 0$ on $[0, \infty)\times[0, \infty)$ with initial conditions and a nonlinear Neumann boundary condition $u_x(0, t) = (f(u_t(0,t)))_t$ at $x=0$. This problem is an exact reduction of a nonlinear Maxwell problem in electrodynamics. In the case where $f\colon\R\to\R$ is an increasing homeomorphism we study global existence, uniqueness and wellposedness of the initial value problem by the method of characteristics and fixed point methods. We also prove conservation of energy and momentum and discuss why there is no wellposedness in the case where $f$ is a decreasing homeomorphism. Finally we show that previously known time-periodic, spatially localized solutions (breathers) of the wave equation with the nonlinear Neumann boundary condition at $x=0$ have enough regularity to solve the initial value problem with their own initial data.
\end{abstract}

	\maketitle
	
\section{Introduction and main results}~

In this paper we study the initial value problem for the following 1+1-dimensional wave equation with quasilinear boundary condition:
\begin{align} \label{eq:IP} 
	\begin{cases}
		V(x) u_{tt}(x, t) - u_{xx}(x, t) = 0, &x \in [0, \infty), t \in [0, \infty), \\
		u_x(0, t) = (f(u_t(0,t)))_t , & x = 0, t \in [0, \infty), \\
		u(x, t_0) = u_0(x), u_t(x, t_0) = u_1(x), & x \in [0, \infty), t = 0.
	\end{cases}
\end{align}
This initial value problem has two main features: the wave equation on the half-axis $[0,\infty)$ is linear with a space-dependent speed of propagation and the boundary condition at $x=0$ is a rather singular, quasilinear, 2nd-order in time Neumann-condition. We show wellposedness on all time intervals $[0,T]$ with $T>0$, and preservation of energy and momentum.

\medskip

Our interest in \eqref{eq:IP} stems from the fact that it appears in the context of electromagnetics as an exact reduction of a nonlinear Maxwell system. We recall the Maxwell equations in the absence of charges and currents 
\begin{align*}
	\nabla\cdot\mathbf{D}&=0, &\nabla\times\mathbf{E}\,=&-\partial_t\mathbf{B}, &\mathbf{D}=&\varepsilon_0\mathbf{E}+\mathbf{P}(\mathbf{E}), \\
	\nabla\cdot\mathbf{B}&=0, &\nabla\times\mathbf{H}=&\,\partial_t\mathbf{D}, &\mathbf{B}=&\mu_0\mathbf{H} 
\end{align*}
with the electric field $\mathbf{E}$, the electric displacement field $\mathbf{D}$, the polarization field $\mathbf{P}$, the magnetic field $\mathbf{B}$, and the magnetic induction field $\mathbf{H}$. Particular properties of the underlying material are modelled by the specification of the relations between $\mathbf{E}, \mathbf{D}, \mathbf{P}$ on one hand, and $\mathbf{B}, \mathbf{H}$ on the other hand. Here, we assume a magnetically inactive material, i.e., $\mathbf{B}=\mu_0\mathbf{H}$, but on the electric side we assume a material with a Kerr-type nonlinear behaviour, cf. \cite{agrawal}, Section~2.3, given through 
$$
\mathbf{P}(\mathbf{E})=\varepsilon_0\chi_1(\mathbf{x})\mathbf{E}+\varepsilon_0\chi_{\text{NL}}(\mathbf{x})g(\abs{\mathbf{E}}^2)\mathbf{E}
$$ 
with $\mathbf{x}=(x,y,z)\in\R^3$ and $|\cdot|$ the Euclidean norm on $\R^3$. For simplicity we assume that $\chi_1, \chi_{\text{NL}}$ are given scalar valued functions instead of the more general situation where they are matrix valued. The scalar constants $\varepsilon_0, \mu_0$ are such that $c=(\varepsilon_0\mu_0)^{-1/2}$ is the speed of light in vacuum. Local existence, wellposedness and regularity results for the general nonlinear Maxwell system have been shown on $\R^3$ by Kato \cite{kato} and on domains by Spitz \cite{spitz1, spitz2}.

In its second order formulation the Maxwell system becomes
\begin{align}
	0=\nabla\times\nabla\times\mathbf{E} +\partial_t^2\Bigl(\mu_0\varepsilon_0(1+\chi_1(\mathbf{x}))\mathbf{E}+\mu_0\varepsilon_0\chi_{\text{NL}}(\mathbf{x})g(\abs{\mathbf{E}}^2)\mathbf{E}\Bigr).
	\label{curlcurl}
\end{align}
We assume additionally that $\chi_1(\mathbf{x})=\chi_1(x)$, $\chi_{\text{NL}}(\mathbf{x})=\chi_{\text{NL}}(x)$ and that $\mathbf{E}$ takes the form of a polarized traveling wave
\begin{equation} \label{polar}
\mathbf{E}(\mathbf{x},t)=(0,0,U(x,\kappa^{-1}y-t))^T.
\end{equation}
Then the quasilinear vectorial wave-type  equation \eqref{curlcurl} turns into the scalar equation 
\begin{equation} \label{second_order_U}
V(x) U_{tt} - U_{xx} + \Gamma(x) (g(U^2)U)_{tt}=0
\end{equation}
for $U=U(x,t)$, where $V(x)=\mu_0\varepsilon_0(1+\chi_1(x))-\kappa^{-2}$ and $\Gamma(x)=\mu_0\varepsilon_0\chi_{\text{NL}}(x)$. Note that \eqref{second_order_U} is an exact reduction of the Maxwell problem, from which all fields can be reconstructed. E.g., the magnetic induction $\mathbf{B}$ can be retrieved from $\nabla\times\mathbf{E}=-\partial_t\mathbf{B}$ by time-integration  and it will satisfy $\nabla\cdot\mathbf{B}=0$ provided it does so at time $t=0$. By assumption the magnetic field is given by $\mathbf{H}=\frac{1}{\mu_0} \mathbf{B}$ and it satisfies $\nabla\times\mathbf{H}=\partial_t\mathbf{D}$. It remains to check that the displacement field $\mathbf{D}$ satisfies the Gauss law $\nabla\cdot\mathbf{D}=0$ in the absence of external charges. This follows directly from the constitutive equation $\mathbf{D}=\varepsilon_0(1+\chi_1(\mathbf{x}))\mathbf{E}+\varepsilon_0\chi_{\text{NL}}(\mathbf{x})g(\abs{\mathbf{E}}^2)\mathbf{E}$ and the assumption of the polarized form of the electric field in \eqref{polar}.  

In the extreme case where $\Gamma(x)=2\delta_0(x)$ is a multiple of the $\delta$-distribution at $0$ and where $U(x,t)=u_t(x,t)$ for an even function $u(x,t)=u(-x,t)$, by removing one time derivative \eqref{second_order_U} becomes 
\begin{align} \label{second_order_u}
	\begin{cases}
		V(x) u_{tt}(x, t) - u_{xx}(x, t) = 0, &x \in [0, \infty), t \in [0, \infty), \\
		u_x(0, t) = (f(u_t(0,t)))_t , & x = 0, t \in [0, \infty)
	\end{cases}
\end{align}
with $f(s) \coloneqq g(s^2)s$. Clearly \eqref{eq:IP} is the initial value problem for \eqref{second_order_u}. 

\medskip

Problem \eqref{second_order_u} with $f(s)=\pm s^3$ has been considered in \cite{kohler_reichel}. Under specific assumptions on the linear potential $V$ the existence of infinitely many breathers, i.e., real-valued, time-periodic, spatially localized solutions of \eqref{second_order_u}, was shown. Typical examples of $V$ were given in classes of piecewise continuous functions having jump discontinuities. Under different assumptions on $V$ and $\Gamma$, but still including $\delta$-distributions, problem \eqref{second_order_u} was considered in \cite{bruell_idzik_reichel} and real-valued breathers were constructed. Our goal is to study the initial value problem \eqref{eq:IP} from the point of view of wellposedness, to derive the conservation of momentum and energy, and to verify that known time-periodic solutions from \cite{kohler_reichel} satisfy \eqref{eq:IP} with their own initial values. Note that the boundary condition in \eqref{eq:IP} becomes $u_x(0,t)=\pm 3 u_t(0,t)^2 u_{tt}(0,t)$ in the model case $f(s)=\pm s^3$. Hence, \eqref{eq:IP} is a singular initial value problem which is not covered by typical theories like, e.g., energy methods or monotone operators. Instead, our approach will be to prove existence by making use of the method of characteristics. Uniqueness, wellposedness, global existence, and the conservation of energy and momentum will build upon this.

\medskip

Our basic assumptions on the initial data $u_0, u_1$ are:
\begin{align}\label{ass:initialdata} \tag{A0} 
	u_0 \in C^1([0, \infty)), 
	\quad 
	u_1 \in C([0, \infty)).
\end{align}\clonelabel{ass:first}{ass:initialdata}%
Here $C^k([0, \infty)) = C^k([0, \infty), \R)$, and in general all function spaces consist of real-valued functions unless the codomain is explicitly mentioned.
Motivated by the results from \cite{kohler_reichel} we are interested in the case where the coefficient $V$ may have discontinuities. In particular, we consider piecewise $C^1$ functions $V$. 

Let $I \subseteq \R$ be a closed interval. We call a function $\phi \colon I \to \R$ piecewise $C^k$ if there exists a discrete set $D \subseteq I$ such that $\phi \in C^k(I \setminus D)$ and the limits $\phi^{(j)}(x-)$ and $\phi^{(j)}(x+)$ exist for all $x \in D(\phi)$ and $0 \leq j \leq k$, although they do not need to coincide. If $I$ is bounded from below (or above), in addition we require $\phi^{(j)}(\min I+)$ (or $\phi^{(j)}(\max I -)$) to exist for all $0 \leq j \leq k$.
Let $PC^k(I)$ denote the set of piecewise $C^k$ functions on $I$, and for $\phi \in PC(I) \coloneqq PC^0(I)$ let us denote by $D(\phi)$ the set of discontinuities of $\phi$. 

For the coefficient $V$ and the nonlinear function $f$ we assume
\begin{align}
	\tag{A1}\label{ass:potential:1} & V \in PC^1([0, \infty)), V, V'\in L^\infty, \inf V > 0, \\
	\tag{A2}\label{ass:potential:2} &\inf\{|d_1-d_2| \text{ with } d_1, d_2\in D(V) \cup \set{0}, d_1\not = d_2\} >0, \\
	\tag{A3}\label{ass:nonlinearity} &f \colon \R \to \R \text{ is an increasing homeomorphism}.
\end{align}\clonelabel{ass:last}{ass:nonlinearity}%
The main theorem of this paper is given next.

\begin{theorem} \label{thm:main}
	Assume \eqref{ass:first}--\eqref{ass:last}. Then \eqref{eq:IP} admits a unique and global $C^1$-solution. Moreover, \eqref{eq:IP} is wellposed on every finite time interval $[0,T]$ with $T>0$.
\end{theorem}

In Proposition~\ref{prop:wp} our concept of continuous dependence on data is stated precisely. In the above result the assumption \eqref{ass:nonlinearity} is crucial. For a decreasing homeomorphism $f$ the result of \cref{thm:main} does not hold, see \cref{rem:decreasing}. Since we have already used the notion of a $C^1$-solution, we are going to explain it in detail next. As the notion of a $C^1$-solution will also be used for subdomains of $[0,\infty)\times [0,\infty)$ we first define the notion of an admissible domain.

\begin{definition}[admissible domain] \label{def:admissible_domain:x}
	We call a set $\Omega \subseteq [0, \infty) \times [0, \infty)$ an \emph{admissible domain} if it is of the form
	\begin{align*}
		\Omega = \{(x, t) \in [0, \infty) \times [0, \infty) \mid t \leq h(x)\}
	\end{align*}
	where $h \equiv +\infty$ or $h \colon [0, \infty) \to \R$ is Lipschitz with $\abs{h_x(x)} \leq \sqrt{V(x)}$ for almost all $x$. We denote the relative interior of $\Omega$ by
	\begin{align*}
		\Omega^\circ \coloneqq \{(x, t) \in [0, \infty) \times [0, \infty) \mid t < h(x)\}.
	\end{align*}
\end{definition}

In order to explain the notion of a $C^1$-solution let us first mention that we cannot expect that a solution of \eqref{eq:IP} has everywhere second derivatives $u_{tt}$ or $u_{xx}$. This is essentially due to the nonlinear boundary condition and the discontinuities of second derivatives which propagate away from $x=0$. However, if we denote by $c(x) \coloneqq \frac{1}{\sqrt{V(x)}}$ the inverse of the $x$-dependent wave speed, then
we can factorize the wave operator as 
\begin{align*}
\frac{\partial^2}{\partial t^2}-c(x)^2 \frac{\partial^2}{\partial x^2} = (\partial_t- c(x)\partial_x)(\partial_t+c(x)\partial_x)+c(x)c'(x)\partial_x.
\end{align*}
It is then reasonable for a $C^1$-solution to have almost everywhere a mixed second directional derivative $\partial^2_{\nu,\mu}$ with directions $\nu=(1,-c(x))$ and $\mu=(1,c(x))$. This is the basis for the following definition.

\begin{definition}[solution] \label{def:solution:x}
	A function $u \in C^1(\Omega)$ on an admissible domain $\Omega$ is called a $C^1$-solution to \eqref{eq:IP} if the following hold:
	\begin{enumerate}[(i)]
		\item For all $(x, t) \in \Omega\setminus(D(c)\cup D(c')\times\R)$ we have $(\partial_t - c(x)\partial_x)(u_t + c(x)u_x)(x, t) = - c(x) c_x(x) u_x(x, t)$.

		\item $(f(u_t(0, t)))_t = u_x(0, t)$ for all $(0, t) \in \Omega^\circ$.

		\item $u(x, 0) = u_0(x)$ for all $(x, 0) \in \Omega$, $u_t(x, 0) = u_1(x)$ for all $(x, 0) \in \Omega^\circ$.
	\end{enumerate}
\end{definition}


Problem \eqref{eq:IP} has a momentum given by
\begin{align} \label{eq:momentum}
	M(u, t) \coloneqq \int_0^\infty V(x) u_t \der x + f(u_t(0, t))	
\end{align} 
and an energy given by
\begin{align} \label{eq:energy} 
	E(u, t) \coloneqq \tfrac12 \int_{0}^{\infty} \left(V(x) u_t(x, t)^2 + u_x(x, t)^2\right) \der x + F(u_t(0, t))
\end{align}
where $F(s) \coloneqq s f(s) - \int_{0}^{s} f(\sigma) \der \sigma$. If, e.g., $f$ is continuously differentiable, then $F(s)$ is a primitive of $s f'(s)$. The conservation of momentum and energy is stated next.


\begin{theorem} \label{thm:conservation}
	Assume \eqref{ass:first}--\eqref{ass:last} and that $u$ is a $C^1$-solution of \eqref{eq:IP} with $u_0'(x), u_1(x) \to 0$ as $x \to \infty$. Then the momentum given by \eqref{eq:momentum} and the energy given by \eqref{eq:energy} are time-invariant.
\end{theorem}

\begin{remark}
	Note that $F(s) = \int_{0}^{s} f(s) - f(\sigma) \der \sigma$ goes to $+\infty$ as $s \to \pm\infty$, so that due to \cref{thm:conservation}, $u_x(\impvar, t)$ and $u_t(\impvar, t)$ are bounded in $L^2(0, \infty)$ and $u_t(0, t)$ is bounded as well.
\end{remark}

Another common notion of solution for \eqref{eq:IP} is the notion of a weak solution, which we only give for $\Omega = [0, \infty)^2$. The fact that a $C^1$-solution to \eqref{eq:IP} is also a weak solution to \eqref{eq:IP} holds true an will be proven in \cref{prop:c1_implies_weak} in \cref{sec:energy}.

\begin{definition}[weak solution] \label{def:weak_solution}
	A function $u \in W^{1, 1}_\text{loc}([0, \infty) \times [0, \infty))$ is called a weak solution to \eqref{eq:IP} if $f(u_t(0, \impvar)) \in L^1_\text{loc}([0, \infty))$, $u(\impvar, 0) = u_0$, and $u$ satisfies
	\begin{align*}
		0 &= \int_0^\infty \int_0^\infty \left(V(x) u_t \varphi_t - u_x \varphi_x\right) \der x \der t
		+ \int_0^\infty f(u_t(0, t)) \varphi_t(0, t) \der t
		\\ &\quad+\int_{0}^{\infty} V(x) u_1(x) \varphi(x, 0) \der x
		+ f(u_1(0)) \varphi(0, 0)
	\end{align*}
	for all $\varphi \in C_c^\infty([0, \infty) \times [0, \infty))$.
\end{definition}

\begin{remark} \label{rem:decreasing}
	Due to assumption \eqref{ass:nonlinearity} we have only considered increasing functions $f$.
	If we instead allow $f \colon \R \to \R$ to be a decreasing homeomorphism, then \eqref{eq:IP} will not be wellposed in general and can have multiple solutions.
	Consider for example the cubic term $f(y) = - y^3$ with constant potential $V = 1$ and homogeneous initial data: 
	\begin{align} \label{eq:loc:decreasing_nonlinearity} 
		\begin{cases}
			u_{tt}(x, t) - u_{xx}(x, t) = 0, &x \in [0, \infty), t \in [0, \infty), \\
			u_x(0, t) = - (u_t(0,t)^3)_t , & x = 0, t \in [0, \infty), \\
			u(x, t_0) = 0, u_t(x, t_0) = 0, & x \in [0, \infty), t = 0.
		\end{cases}
	\end{align}
	By direct calculation one can show that the right-traveling wave
	\begin{align*}
		u_p(x, t) = \begin{cases}
			\left( \frac{2}{3} \left(t - x\right)\right)^{\frac{3}{2}}, &x < t, \\
			0, & x \geq t
		\end{cases}
	\end{align*}
	is a nontrivial solution to \eqref{eq:loc:decreasing_nonlinearity}. In fact, $u$ is a $C^1$-solution of $(\partial_x+\partial_t)u=0$. But \eqref{eq:loc:decreasing_nonlinearity} also has the trivial solution $u = 0$, or $u(x, t) = \pm u_p(x, t - \tau)$ for any $\tau \geq 0$. However, due to the continuity of $f^{-1}$, one can still show existence of solutions to \eqref{eq:IP} in the case where $f$ grows at least linearly, cf. \eqref{ass:linear_growth}. This follows from the arguments in \cref{sec:linear_results,sec:main_proof}. \cref{thm:conservation} also holds when $f$ is decreasing, but now the quantity $F(y)$ tends to $-\infty$ as $y \to \pm \infty$, so that \eqref{eq:energy} does not give rise to estimates on $u$. Lastly, also in this case $C^1$-solutions to \eqref{eq:IP} are weak solutions.
\end{remark}

\medskip

In addition to the problem being posed on the positive real half-line $x \in [0, \infty)$, we can also consider the same quasilinear problem posed on a bounded domain $x \in [0, L]$ where we impose a homogeneous Dirichlet condition at $x=L$:
\begin{align} \label{eq:IP:bdd}
	\begin{cases}
		V(x) u_{tt}(x, t) - u_{xx}(x, t) = 0, &x \in [0, L], t \in [0, \infty), \\
		u_x(0, t) = (f(u_t(0,t)))_t , & t \in [0,\infty), \\
		u(x, 0) = u_0(x), u_t(x, 0) = u_1(x), & x \in [0, L], \\
		u(L, t) = 0, & t \in [0, \infty). 
	\end{cases}
\end{align}

Both \cref{thm:main} and \cref{thm:conservation} remain valid when making the obvious adaptations to this setting.

\begin{theorem} \label{thm:main_and_energy:bdd}
	Assume \eqref{ass:first}--\eqref{ass:last}. Then \eqref{eq:IP:bdd} admits a unique and global $C^1$-solution $u$. Moreover, the energy given by 
	\begin{align*}
		E(u, t) \coloneqq \tfrac12 \int_{0}^{L} \left(V(x) u_t(x, t)^2 + u_x(x, t)^2\right) \der x + F(u_t(0, t)).
	\end{align*}
	is time-invariant.
\end{theorem}

\begin{remark} 
    For Dirichlet boundary data, momentum is in general not conserved. 
\end{remark}

The paper is structured as follows. In \cref{sec:x_to_z} we provide a change of variables which turns the wave operator with variable wave speed in \eqref{eq:IP} into a constant coefficient operator with a convenient factorization. In \cref{sec:linear_results} we collect all results on the linear wave equation that is obtained from the change of variables in \cref{sec:x_to_z}. \cref{sec:main_proof} contains the proof of the existence and uniqueness part of the main result of \cref{thm:main} under an extra assumption which will removed in the subsequent \cref{sec:energy}. This section also contains the proof of energy and momentum conservation as stated in Theorem~\ref{thm:conservation}, and the fact that $C^1$-solutions of \eqref{eq:IP} in the sense of Definition~\ref{def:solution:x} are also weak solutions, cf. \cref{prop:c1_implies_weak}. The wellposedness part of \cref{thm:main} can be found in \cref{sec:wellposedness}. Finally, in \cref{sec:regularity} we verify that the breather solutions from \cite{kohler_reichel} satisfy \eqref{eq:IP} with their own initial values. The \cref{appendix_A,sec:sobolevSpaces} contain some technical results used in the proofs of the main results.

\section{A change of variables}
\label{sec:x_to_z}

It will be convenient to normalize the wave speed to $1$. To achieve this, we introduce a new variable $z = \kappa(x) = \int_{0}^{x} \frac{1}{c(s)} \der s$, and thus a new coordinate system $(z, t)$. Avoiding new notation we denote the functions $V, c, u, u_0, u_1$ transformed into this new coordinate system again by $V, c, u, u_0, u_1$. The relation between the two coordinate systems is given by
\begin{align*}
	\pdv{z}{x} = \frac{1}{c(x)} \quad\text{or}\quad c(x) \partial_x = \partial_z  \quad\text{or}\quad \der x = c(x) \der z.
\end{align*}
From now on until the end of \cref{sec:energy}, we will exclusively work with the coordinate system $(z, t)$. As before we denote the points where $c$ is discontinuous by $D(c)$ and the points where $c_z$ is discontinuous by $D(c_z)$.

Formally the initial value problem \eqref{eq:IP} transforms into
\begin{align} \label{eq:IP_z} 
	\begin{cases}
		u_{tt}(z, t) - u_{zz}(z, t) = -\frac{c_z(z)}{c(z)}u_z(z,t), &z \in [0, \infty), t \in [0, \infty), \\
		\frac{1}{c(0)}u_z(0, t) = (f(u_t(0,t)))_t , & t \in [0, \infty), \\
		u(z,0) = u_0(z), u_t(z, 0) = u_1(z), & z \in [0, \infty).
	\end{cases}
\end{align}

where we need to take into account that $u_x=\frac{1}{c} u_z$ is continuous (and not $u_z$ itself) and that the differential equation does not hold at the discontinuities of $c$ and $c_z$. A detailed definition of the solution concept is given below in \cref{def:solution:z}.

We begin by rephrasing \cref{def:admissible_domain:x,def:solution:x} for the new coordinate system.

\begin{definition}[admissible domain] \label{def:admissible_domain:z}
	We call a set $\Omega \subseteq [0, \infty) \times [0, \infty)$ an \emph{admissible domain} if it is of the form
	\begin{align*}
		\Omega = \{(z, t) \in [0, \infty) \times [0, \infty) \mid t \leq h(z)\}
	\end{align*}
	where $h \equiv +\infty$ or $h \colon [0, \infty) \to \R$ is Lipschitz continuous with Lipschitz constant $1$. We denote its relative interior by
	\begin{align*}
		\Omega^\circ \coloneqq \{(z, t) \in [0, \infty) \times [0, \infty) \mid t < h(z)\}.
	\end{align*}
\end{definition}

Next we introduce function spaces that capture the condition of the continuity of $\frac{1}{c} u_z$.
\begin{definition}[$x$-dependent function spaces]
	Let the transformation between $(x, t)$ and $(z, t)$-coordinates be given by $\tilde \kappa(x, t) \coloneqq (\kappa(x), t)=(z,t)$. For $\Omega \subseteq [0, \infty) \times [0, \infty)$ we write
	\begin{align*}
		\spacext{\Omega} \coloneqq \{u \colon \Omega \to \R \mid u \circ \tilde \kappa \in C^1(\tilde \kappa^{-1}(\Omega))\}
	\end{align*}
	where we understand $u$ to be a function of $(z, t)$ variables, and $\tilde u \coloneqq u \circ \tilde \kappa$ is the $(x, t)$-dependent version of $u$, i.e. $\tilde u(x, t) = u(z, t)$ holds.
	Note that $u \in C^1_{(x, t)}(\Omega)$ if and only if $u, u_t, \frac{1}{c}u_z \in C(\Omega)$.

	Similarly, for an interval $I \subseteq [0, \infty)$ we define
	\begin{align*}
		\spacex{I} \coloneqq \{v \colon I \to \R \mid v \circ \kappa \in C^1(\kappa^{-1}(I)) \}.
	\end{align*}
	where again we understand $v$ to be a function of $z$.
\end{definition}

\begin{definition}[solution] \label{def:solution:z}
	A function $u \in \spacext{\Omega}$ on an admissible domain $\Omega$ is called a $C^1$-solution to \eqref{eq:IP_z} if the following hold:
	\begin{enumerate}[(i)]
		\item For all $(z, t) \in \Omega \setminus (D(c)\cup D(c_z) \times \R)$ we have $(\partial_t - \partial_z)(u_t + u_z)(z, t) = - \frac{c_z(z)}{c(z)} u_z(z, t)$.

		\item $f(u_t(0,t))_t=\frac{1}{c(0)} u_z(0, t)$ for all $(0, t) \in \Omega^\circ$.

		\item $u(z, 0) = u_0(z)$ for all $(z, 0) \in \Omega$, $u_t(z, 0) = u_1(z)$ for all $(z, 0) \in \Omega^\circ$.
	\end{enumerate}
\end{definition}

\begin{remark} Note that $u \colon \Omega\to\R$ is a $C^1$-solution to \eqref{eq:IP} in the $(x,t)$-coordinates if and only if it is a $C^1$-solution to \eqref{eq:IP_z} in the $(z,t)$-coordinates.
\end{remark}

\section{Auxiliary results on the linear part}\label{sec:linear_results}

In this section we gather some auxiliary results and estimates on the linear wave equation. These will prove useful for the study of the nonlinear boundary condition. All results of this section hold under the assumptions (A0)--(A3).

We first note that the wave equation has finite speed of propagation; if we know its behavior at time $t_0$ on an interval $[z_0 - r, z_0 + r]$, then we can defer its accurate behavior on the space-time triangle with corners $(z_0 - r, t_0)$, $(z_0 + r, t_0)$ and $(z_0, t_0 + r)$.

\begin{definition}
	For $(z_0, t_0) \in \R^2$ and $r > 0$ we denote the triangle with corners $(z_0 - r, t_0)$, $(z_0 + r, t_0)$ and $(z_0, t_0 + r)$ by
	\begin{align*}
		\tri(z_0, t_0, r) \coloneqq \{(z, t) \in \R^2 \mid t \geq t_0, \abs{z - z_0} + \abs{t - t_0} \leq r\},
	\end{align*}
	its base projected onto the $z$-axis is given by $P_z \tri(z_0, t_0, r) = [z_0 - r, z_0 + r]$ with projection $P_z(z, t) \coloneqq z$. Similarly, we define left and right half triangles 
	\begin{align*}
		&\tri_-(z_0, t_0, r) \coloneqq \tri(z_0, t_0, r)\cap \{z\leq z_0\}, &&\tri_+(z_0, t_0, r) \coloneqq \tri(z_0, t_0, r)\cap \{z\geq z_0\}	
	\intertext{whose bases are given by} 
		&P_z \tri_-(z_0, t_0, r) = [z_0 - r, z_0], &&P_z \tri_-(z_0, t_0, r) = [z_0, z_0 + r].
	\end{align*}
\end{definition}

Recall the solution formula for the 1-dimensional wave equation:

\begin{theorem}\label{thm:wave_solution_formula}
	Let $(z_0, t_0) \in \R^2$, $r > 0$, $\tri \coloneqq \tri(z_0, t_0, r)$ and $B \coloneqq P_z \tri$. Assume that $u_0 \in C^1(B)$, $u_1 \in C(B)$, and $g\in L^\infty(\tri)$ is continuous outside a set $L$ consisting of finitely many lines of the form $\{z = \text{const}\}$. Then the function
	\begin{align*}
		u(z, t) 
		= \tfrac12\left(u_0(z + t - t_0) + u_0(z - t + t_0)\right) 
		+ \tfrac12 \int_{z-t+t_0}^{z+t-t_0} u_1(y) \der y 
		+ \tfrac12 \int_{\tri(z, t_0, t - t_0)} g(y, \tau) \der (y, \tau)
	\end{align*}
	belongs to $C^1(\tri)$ and is the unique $C^1$-solution of the problem
	\begin{align*} 
		\begin{cases}
			(\partial_t - \partial_z) (u_t + u_z) = g, & (z, t) \in \tri, \\
			u(z, t_0) = u_0(z), \quad u_t(z, t_0) = u_1(z), & z \in B
		\end{cases}
	\end{align*}
	in the following sense: $u(\cdot,t_0)=u_0(\cdot)$, $u_t(\cdot,t_0)=u_1(\cdot)$ on $B$ and the directional derivative $(\partial_t - \partial_z) (u_t + u_z)$ exists and equals $g$ on $\tri^\circ\setminus L$.	
\end{theorem}

\begin{remark} \label{rem:opposite_factorization} 
    For every $C^1$-solution $u$ of $(\partial_t - \partial_z) (u_t + u_z) = g$ on a domain we have that $(\partial_t + \partial_z) (u_t - u_z) = (\partial_t - \partial_z) (u_t + u_z)$ wherever $g$ is continuous, cf. Schwarz's theorem in \cite[Theorem 9.41]{rudin_analysis}. As a consequence, any of the two factorizations of the wave operator $(\partial_t-\partial_z)(\partial_t+\partial_z)$ or $(\partial_t+\partial_z)(\partial_t-\partial_z)$ can be used and yields the same solution. 
\end{remark}

By combining the above \cref{thm:wave_solution_formula} with a fixed point argument, we can treat the initial value problem for $\left(\partial_t - \partial_z\right) \left(u_t + u_z\right) = - \frac{c_z(z)}{c(z)} u_z$ on sufficiently small triangles $\tri$. In order to have a slightly more general situation available we work with a piecewise continuous function $\lambda$ instead of $\frac{c_z}{c}$. 

\begin{corollary}\label{cor:linear_uniqueness}
	Let $(z_0, t_0) \in \R^2$ and $\tri \coloneqq \tri(z_0, t_0, r)$, $B \coloneqq P_z \tri$ for $r>0$. Assume $u_0 \in C^1(B)$, $u_1 \in C(B)$ and $\lambda \in PC(B)$ such that $r \norm{\lambda}_\infty < 1$. Then
	\begin{align} \label{local_problem}
		\begin{cases}
			(\partial_t - \partial_z) (u_t + u_z) = -\lambda(z) u_z, & (z, t) \in \tri,\\
			u(z, t_0) = u_0(z), u_t(z, t_0) = u_1(z), & z \in B
		\end{cases}
	\end{align} 
	has a unique solution $u\in C^1(\tri)$ in the sense of Theorem~\ref{thm:wave_solution_formula} with $g=-\lambda u_z$ and $L=D(\lambda)\times\R$. We denote this solution by $\Phi(u_0, u_1) \coloneqq u$. 
\end{corollary}

\begin{remark} \label{rem:symmetry} 
		If additionally $u_0, u_1$ are odd around $z=z_0$ and $\lambda$ is odd around $z=z_0$, then the solution of \eqref{local_problem} is odd around $z=z_0$. To see this, notice that under these assumptions the odd reflection of the solution $u$ of \eqref{local_problem} again solves \eqref{local_problem} -- but with the opposite factorization of the wave operator. Hence, by \cref{rem:opposite_factorization} and uniqueness of solutions, $u$ coincides with its odd reflection.
\end{remark}

\begin{proof}[Proof of \cref{cor:linear_uniqueness}]
	W.l.o.g. we assume $(z_0, t_0) = (0, 0)$. Let $u \in C^1(\tri)$. Then by \cref{thm:wave_solution_formula} $u$ is a solution if and only if
	\begin{align} \label{eq:loc:solution_representation}
		u(z, t) 
		= \tfrac12\left(u_0(z + t) + u_0(z - t)\right) 
		+ \tfrac12 \int_{z-t}^{z+t} u_1(y) \der y 
		- \tfrac12 \int_{\tri(z, 0, t)} \lambda(y) u_z(y, \tau) \der (y, \tau)
	\end{align}
	holds for $(z, t) \in \tri$.
	Taking the derivative w.r.t. $z$ we obtain
	\begin{equation} \label{eq:linear_fixedpoint_derived}
		\begin{split}
			u_z(z, t) 
			&= \tfrac12\left(u_0'(z + t) + u_0'(z - t)\right) 
			+ \tfrac12 \left(u_1(z + t) - u_1(z - t)\right) \\
			&\quad- \tfrac12 \int_{0}^{t} \lambda(z + t - s) u_z(z + t - s, s) \der s + \tfrac12 \int_{0}^{t} \lambda(z - t + s) u_z(z - t + s, s) \der s.
		\end{split}
	\end{equation}
	We consider \eqref{eq:linear_fixedpoint_derived} as a fixed point problem for $u_z\in C(\tri)$. If we denote the right-hand side of \eqref{eq:linear_fixedpoint_derived} by $T(u_z)(z, t)$, then clearly $T$ maps $C(\tri)$ into itself. Furthermore, one has
		\begin{align*}
			&\norm{T(u_z) - T(w_z)}_\infty  \\
			&\quad= \tfrac12 \sup_{(z, t)\in\tri} \abs{-\int_{0}^{t} \lambda(z + s)\,[u_z - w_z](z + s, t - s)\der s + \int_{0}^{t} \lambda(z - s) \,[u_z - w_z](z - s, t - s) \der s} \\
			&\quad \leq \norm{\lambda}_\infty r \cdot \norm{u_z - w_z}_\infty
		\end{align*}
	so that by Banach's fixed-point theorem there exists a unique solution $u_z$ of \eqref{eq:linear_fixedpoint_derived}. With the help of $u_z$ we define $u$ as in \eqref{eq:loc:solution_representation} and thus get the claimed result. 
\end{proof}

In the setting of the above proof, we can obtain estimates on the solution $u$.
First, if we set $q\coloneqq r \norm{\lambda}_\infty$, then by Banach's fixed-point theorem we have
\begin{align*}
	\norm{u_z - 0}_\infty \leq \frac{1}{1 - q} \norm{T(0) - 0}_\infty.
\end{align*}

Using $\norm{T(0)}_\infty \leq \norm{u_0'}_\infty + \norm{u_1}_\infty$, we obtain
\begin{align*}
	\norm{u_z}_\infty \leq \frac{1}{1 - q} \left(\norm{u_0'}_\infty + \norm{u_1}_\infty\right)
\end{align*}
From 
\begin{align*}
	u(z, t) 
	&= \tfrac12\left(u_0(z + t) + u_0(z - t)\right) 
	+ \tfrac12 \int_{z - t}^{z + t} u_1(y) \der y 
	- \tfrac12 \int_{0}^{t} \int_{z - (t - \tau)}^{z + (t - \tau)} \lambda(y) u_z(y, \tau) \der y \der \tau, \\
	u_t(z, t) 
	&= \tfrac12\left(u_0'(z + t) - u_0'(z - t)\right) 
	+ \tfrac12 \left(u_1(z + t) + u_1(z - t)\right) \\
	&\quad- \tfrac12 \int_{0}^{t} \lambda(z + s) u_z(z + s, t - s) \der s - \tfrac12 \int_{0}^{t} \lambda(z - s) u_z(z - s, t - s) \der s
\end{align*}
we also obtain
\begin{align*}
	\norm{u}_\infty 
	\leq \norm{u_0}_\infty + r \norm{u_1}_\infty + \tfrac{1}{2}r^2 \norm{\lambda}_\infty \norm{u_z}_\infty, 
	\qquad
	\norm{u_t}_\infty 
	\leq \norm{u_0'}_\infty + \norm{u_1}_\infty + r \norm{\lambda}_\infty \norm{u_z}_\infty.
\end{align*}
Combining these estimates, we get the following result.

\begin{corollary}\label{cor:operator_estimates} In the setting of \cref{cor:linear_uniqueness}, the following estimates hold with $q \coloneqq r\norm{\lambda}_\infty$:
	\begin{align*}
		\norm{u}_\infty &\leq \norm{u_0}_\infty + \frac{r q}{2(1 - q)} \norm{u_0'}_\infty + \frac{r(1-\frac{1}{2}q)}{1 - q} \norm{u_1}_\infty, \\
		\norm{u_z}_\infty &\leq \frac{1}{1 - q} \left(\norm{u_0'}_\infty + \norm{u_1}_\infty\right), \\
		\norm{u_t}_\infty &\leq \frac{1}{1 - q} \left(\norm{u_0'}_\infty + \norm{u_1}_\infty\right).
	\end{align*}
	In particular, there exists a constant $C = C(r, \norm{\lambda}_\infty)$ such that the operator-norm of the linear solution operator $\Phi:C^1(B) \times C(B) \to C^1(\tri)$, which maps the data $(u_0,u_1)\in C^1(B) \times C(B)$ to the solution of \eqref{local_problem}, satisfies
	\begin{align*}
		\norm{\Phi} \leq C.
	\end{align*}
\end{corollary}

Recall that in \cref{def:solution:z} we required $\frac{u_z}{c}$ to be continuous. Since $c$ may have jumps, e.g. at $z_0$, we also need to treat the jump condition
\begin{align*}
	\frac{u_z(z_0+, t)}{c(z_0+)} = \frac{u_z(z_0-, t)}{c(z_0-)}.
\end{align*}
We prepare this in the following lemma by adding to \eqref{local_problem} the inhomogeneous Dirichlet condition $u(z_0, t) \overset{!}{=} b(t)$ at the spatial boundary $z=z_0$. 

\begin{lemma} \label{lem:dirichlet_phi}
	Let $(z_0, t_0) \in \R^2$ and $\tri_+ \coloneqq \tri_+(z_0, t_0, r)$, $B_+ \coloneqq P_z \tri_+$ for $r > 0$. Assume $u_0 \in C^1(B_+)$, $u_1 \in C(B_+), b \in C^1([t_0, t_0 + r])$ with $b(t_0) = u_0(z_0), b'(t_0) = u_1(z_0)$ and $\lambda \in PC(B_+)$ such that $r \norm{\lambda}_\infty < 1$. Then the problem
	\begin{align} \label{eq:pb:dirichlet_phi}
		\begin{cases}
			(\partial_t - \partial_z) (u_t + u_z) = - \lambda(z) u_z, & (z, t) \in \righttri^\circ, \\
			u(z_0, t) = b(t), & t \in [t_0, t_0 + r], \\
			u(z, t_0) = u_0(z), u_t(z, t_0) = u_1(z), & z \in B_+,
		\end{cases}
	\end{align} 
	has a unique $C^1$-solution $u \colon \righttri \to \R$ in the sense of \cref{thm:wave_solution_formula} with $g=-\lambda u_z$ and $L=D(\lambda)\times\R$. We denote this solution by $\Phi_+(b, u_0, u_1) \coloneqq u$. The assertion also holds for the right half triangle $\Delta_-\coloneqq \tri_-(z_0, t_0, r)$ with corresponding solution operator $\Phi_-$.
\end{lemma}

\begin{proof}
	Note that the function $G^b$ defined on $\righttri$ by
	\begin{align}\label{eq:def:homogenous_boundary_solution}
		G^b(z, t) = \begin{cases}
			b(t_0) + (t - t_0) b'(t_0), & z - z_0 > t - t_0 \geq 0, \\
			b(t + z_0 - z) + (z - z_0) b'(t_0), &t - t_0 \geq z - z_0 \geq 0 \\
		\end{cases}
	\end{align}
	belongs to $C^1(\righttri)$, solves the homogenous wave equation $(\partial_t - \partial_z)(\partial_t + \partial_z) G^b = 0$ on $\righttri$, and satisfies $G^b(z_0, t) = b(t)$. Setting $v \coloneqq u - G^b$, problem \eqref{eq:pb:dirichlet_phi} can be rewritten as
	\begin{align} \label{eq:loc:new_problem}
		\begin{cases}
			(\partial_t - \partial_z) (v_t + v_z) = -\lambda(z)\left(v_z+G^b_z\right), & (z, t) \in \righttri^\circ, \\
			v(z_0, t) = 0, & t \in [t_0, t_0 + r], \\
			v(z, t_0) = u_0(z) - b(t_0) \eqqcolon v_0(z), &z \in B_+, \\
			v_t(z, t_0) = u_1(z) - b'(t_0) \eqqcolon v_1(z), &z \in B_+.
		\end{cases}
	\end{align}
	Note that $v_0(z_0) = v_1(z_0) = 0$ by assumption.
	If we extend the functions $v_0$, $v_1$, and $\lambda$ in an odd way and $G^b$ in an even way around $z=z_0$, we can consider the problem
	\begin{align} \label{eq:loc:odd_extension}
		\begin{cases}
			\left(\partial_t - \partial_z\right) (\tilde v_t + \tilde v_z) =  -\lambda_\text{odd}(z) \cdot \left(\tilde v_z + G^b_{\text{even},z}\right) & (z, t) \in \tri^\circ, \\
			\tilde v(z, t_0) = v_{0, \text{odd}}(z), &z \in B, \\
			\tilde v_t(z, t_0) = v_{1, \text{odd}}(z), &z \in B,
		\end{cases}
	\end{align}
	where $\tri \coloneqq \tri(z_0, t_0, r)$ and $B \coloneqq P_z \tri$. Arguing as in the proof of \cref{cor:linear_uniqueness}, we see that due to the Banach fixed-point theorem, \eqref{eq:loc:odd_extension} has a unique solution, which must be odd, cf. \cref{rem:symmetry}. Now, on one hand the solution of \eqref{eq:loc:odd_extension} solves (after restriction to $\tri_+$) \eqref{eq:loc:new_problem} and, on the other hand, after odd extension around $z=z_0$ every solution of \eqref{eq:loc:new_problem} solves \eqref{eq:loc:odd_extension}. This shows existence and uniqueness for \eqref{eq:loc:new_problem} and hence for \eqref{eq:pb:dirichlet_phi}.
\end{proof}

\begin{remark} \label{rem:estimate:boundary_solution_op}
	One can show that there exists a constant $C = C(r, \norm{\lambda}_{\infty})$ such that 
	\begin{align*}
		\Phi_\pm: C^1([t_0,t_0+r])\times C^1(B_\pm) \times C(B_\pm) \to C^1(\tri_{\pm})	
	\end{align*}
	satisfy $\|\Phi_\pm\|\leq C$. 
\end{remark}

When treating the nonlinear problem \eqref{eq:IP}, the operators $\Phi_{\pm}$ play an important role and the estimate in \cref{rem:estimate:boundary_solution_op} will be used. However, we need to investigate the dependency of $\Phi_{\pm}$ on the datum $b$ more precisely. This will be achieved next in the case where $u_0 = u_1 = 0$.

\begin{lemma}[Estimate on $\Phi_\pm$ in the case $u_0=u_1=0$]
	\label{lem:fine_boundary_estimate}
	Let $\tri_\pm$, and $\lambda$ be as in \cref{lem:dirichlet_phi} with $q \coloneqq r \norm{\lambda}_\infty < 1$. Assume $b\in C^1([t_0,t_0+r])$ and $b(t_0)=b'(t_0)=0$. Then for $u \coloneqq \Phi_\pm(b, 0, 0)$ one has
    \begin{align*}
    	\abs{u_z(z, t) \pm b'(m)} \leq  \alpha\abs{z - z_0} \abs{b'(m)} + \beta \int_{t_0}^{m} \abs{b'(\tau)} \der \tau,
	\end{align*}
	where $m \coloneqq \max\set{t_0, t - \abs{z - z_0}}, \alpha \coloneqq \frac{2}{4 - q} \norm{\lambda}_\infty$, and $\beta \coloneqq \frac{4}{(2-q)(4-q)} \norm{\lambda}_\infty$.
\end{lemma}

\begin{proof} 
	We only give the proof in the ``$+$''-case and for $(z_0,t_0)=(0,0)$. We revisit the proof of \cref{lem:dirichlet_phi} where $\Phi_+$ is defined. From \eqref{eq:linear_fixedpoint_derived} we know that $v_z$ satisfies
	\begin{align*}
		v_z(z, t) 
		= &-\tfrac12 \int_{0}^{t} \lambda_\text{odd}(z + s) \cdot \left(G^b_{\text{even}, z}(z + s, t - s) + v_z(z + s, t - s)\right) \der s \\
		& +  \tfrac12 \int_{0}^{t} \lambda_\text{odd}(z - s) \cdot \left(G^b_{\text{even}, z}(z - s, t - s) + v_z (z - s, t - s)\right)  \der s.
	\end{align*}
	We denote the term on the right-hand side by $T(v_z)(z, t)$ and already know that $T$ is Lipschitz continuous with constant $q < 1$. Therefore we may write the solution as $v_z \coloneqq \lim\limits_{n \to \infty} T^n(0)$ and thus have to study $v_z^{(n)} \coloneqq T^n(0)$. The claimed inequality for $u_z$ will follow once we have shown that
	\begin{align*}
		|v_z(z, t)| \leq \alpha \abs{z - z_0} \abs{b'(m)} + \beta \int_{t_0}^{m} \abs{b'(\tau)} \der \tau.
	\end{align*}
	Due to $v_z \coloneqq \lim\limits_{n \to \infty} T^n(0)$ it is sufficient to show that this estimate holds for all $v_z^{(n)}$. Since $v_z^{(0)} = 0$, there is nothing left to show for $n = 0$. Now assume that the estimate has been shown for some fixed $n$. 
	Recalling the definition of $G^b$ from \eqref{eq:def:homogenous_boundary_solution}, we have
	\begin{align*}
		G_{\text{even},z}^b(z, t) = - \sign(z) b'(\max\{t - \abs{z}, 0\}).
	\end{align*}
	Notice that  $G^b_{\text{even}, z}(z,t)$ vanishes for $|z|\geq t$. Therefore, if $v_z^{(n)}$ vanishes for $|z|\geq t$ then also $v_z^{(n+1)} = T(v_z^{(n)})$ vanishes on this set.
	So in the following we may assume $\abs{z} < t$. We will only consider $z \geq 0$ as $z < 0$ can be treated similarly.
	For $z\geq 0$ and $t>z$ the expression $m = \max\{t - \abs{z}, 0\}$ simplifies to $ m = t - z$.
	We begin by estimating the terms which are independent of $v_z^{(n)}$:
	\begin{align*}
		&\abs{\int_{0}^{t} \lambda_\text{odd}(z + s) G^b_{\text{even},z}(z + s, t - s) \der s} \\
		&\quad= \abs{-\int_{0}^{t} \lambda_\text{odd}(z + s) b'(\max\{t - z - 2 s, 0\}) \der s} \\
		&\quad\leq \tfrac12 \norm{\lambda}_\infty \int_{0}^{t - z} \abs{b'(\tau)} \der \tau= \tfrac12 \norm{\lambda}_\infty \int_{0}^{m} \abs{b'(\tau)} \der \tau, \\
        &\abs{\int_{0}^{t} \lambda_\text{odd}(z - s) G^b_{\text{even},z}(z - s, t - s) \der s} \\
		&\quad= \abs{- \int_{0}^{z} \lambda_\text{odd}(z - s) b'(t - z) \der s
			+ \int_{z}^{t} \lambda_\text{odd}(z - s) b'(\max\{t + z - 2 s, 0\}) \der s} \\
		&\quad\leq \norm{\lambda}_\infty z \abs{b'(t - z)} + \tfrac12 \norm{\lambda}_\infty \int_0^{t - z} \abs{b'(\tau)} \der \tau 
			= \norm{\lambda}_\infty z \abs{b'(m)} + \tfrac12 \norm{\lambda}_\infty \int_0^{m} \abs{b'(\tau)} \der \tau.
	\intertext{The remaining two summands are treated by}
		&\abs{\int_{0}^{t} \lambda_\text{odd}(z + s) v^{(n)}_z(z + s, t - s) \der s} \\
		&\quad\leq \norm{\lambda}_\infty \int_0^t \left(\alpha (z + s) \abs{b'(\max\{t - z - 2 s, 0\})} + \beta \int_{0}^{\max\{t - z - 2 s, 0\}} \abs{b'(\tau)} \der \tau\right) \der s \\
		&\quad= \norm{\lambda}_\infty \int_0^{\frac{t - z}{2}} \left(\alpha (z + s) \abs{b'(t - z - 2 s)} + \beta \int_{0}^{t - z - 2 s} \abs{b'(\tau)} \der \tau\right) \der s \\
		&\quad\leq \norm{\lambda}_\infty \int_0^{\frac{t - z}{2}} \left(\alpha\frac{t + z}{2} \abs{b'(t - z - 2 s)} + \beta \int_{0}^{t - z} \abs{b'(\tau)} \der \tau\right) \der s \\
		&\quad= \norm{\lambda}_\infty \left(\alpha \frac{t + z}{4} + \beta \frac{t - z}{2}\right)\int_{0}^{m} \abs{b'(\tau)} \der \tau, \\
		&\abs{\int_{0}^{t} \lambda_\text{odd}(z - s) v^{(n)}_z(z - s, t - s) \der s} \\
		&\quad\leq \norm{\lambda}_\infty \int_0^t \left(\alpha \abs{z - s} \abs{b'(\max\{t - s - \abs{z - s}, 0\})} + \beta \int_{0}^{\max\{t - s - \abs{z - s}, 0\}} \abs{b'(\tau)} \der \tau\right) \der s \\
		&\quad= \norm{\lambda}_\infty \int_0^z \left(\alpha (z - s) \abs{b'(t - z)} + \beta \int_{0}^{t - z} \abs{b'(\tau)} \der \tau\right) \der s \\
		&\qquad+ \norm{\lambda}_\infty \int_z^{\frac{z + t}{2}} \left(\alpha (s - z) \abs{b'(t + z- 2 s)} + \beta \int_{0}^{t + z - 2 s} \abs{b'(\tau)} \der \tau\right) \der s \\
		&\quad\leq \norm{\lambda}_\infty \left(\alpha \frac{z^2}{2} \abs{b'(m)} + \beta z \int_{0}^{m} \abs{b'(\tau)} \der \tau\right) \\
		&\qquad+ \norm{\lambda}_\infty \left(\alpha \frac{t - z}{4} + \beta \frac{t - z}{2}\right) \int_{0}^{m} \abs{b'(\tau)} \der \tau.
	\end{align*}
	Summing up all four estimates, we obtain
	\begin{align*}
		&2 \abs{v^{(n+1)}_z(z, t)} \\
		&\quad\leq \tfrac12 \norm{\lambda}_\infty \int_0^{m} \abs{b'(\tau)} \der \tau \\
		&\qquad+ \norm{\lambda}_\infty z \abs{b'(m)} + \tfrac12 \norm{\lambda}_\infty \int_0^{m} \abs{b'(\tau)} \der \tau \\
		&\qquad+ \norm{\lambda}_\infty \left(\alpha \frac{t + z}{4} + \beta \frac{t - z}{2}\right)\int_{0}^{m} \abs{b'(\tau)} \der \tau \\
		&\qquad+ \norm{\lambda}_\infty \left(\alpha \frac{z^2}{2} \abs{b'(m)} + \beta z \int_{0}^{m} \abs{b'(\tau)} \der \tau\right) \\
		&\qquad+ \norm{\lambda}_\infty \left(\alpha \frac{t - z}{4} + \beta \frac{t - z}{2}\right) \int_{0}^{m} \abs{b'(\tau)} \der \tau \\
		&\quad= \norm{\lambda}_\infty \left(1 + \alpha \frac{z}{2}\right) z \abs{b'(m)} \\
		&\qquad+ \norm{\lambda}_\infty \left(\tfrac12 + \tfrac12 + \alpha \frac{t + z}{4} + \beta \frac{t - z}{2} + \beta z + \alpha \frac{t - z}{4} + \beta \frac{t - z}{2}\right) \int_{0}^{m} \abs{b'(\tau)}\\
		&\quad\eqqcolon 2 C_1 z \abs{b'(m)}+ 2 C_2 \int_{0}^{m} \abs{b'(\tau)} \der \tau.
	\end{align*}
	It remains to verify $C_1 \leq \alpha$ and $C_2 \leq \beta$. In fact, using $t, z \leq r$, we obtain
	\begin{align*}
		&2 C_1 \leq \norm{\lambda}_\infty + \frac{q}{2} \alpha = 2 \alpha, \\
		&2 C_2 \leq \norm{\lambda}_\infty + \frac{q}{2} \alpha + q \beta = 2\alpha + q\beta = 2 \beta,
	\end{align*}
	where the equalities hold by definition of $\alpha$ and $\beta$, respectively.
\end{proof}

\section{Proof of \cref{thm:main}} \label{sec:main_proof}

In this section, we will prove the existence and uniqueness part of the main \cref{thm:main} under the additional assumption that $f$ grows at least linearly, i.e., for some $A, B > 0$ we have
\begin{align}\label{ass:linear_growth} \tag{A4}
	\abs{f(x)} \geq A \abs{x} - B \quad \text{for } x \in \R.
\end{align} 
In \cref{sec:energy} we will show how to remove this assumption. The wellposedness part of \cref{thm:main} will be completed in Section~\ref{sec:wellposedness}. 

We will again use that the wave equation has finite speed of propagation so that we may argue locally. To be more specific, we will work on the following types of triangular domains:
\begin{itemize}
	\item A \emph{jump triangle} is a triangle $\tri = \tri(z_0, 0, r)$ with base $B = P_z \tri \subseteq (0, \infty)$, where $z_0 \in D(c)$ and $B$ intersects $D(c)$ in no other point. These are useful for the study of the jump condition $\frac{u_z(z+, t)}{c(z+)} = \frac{u_z(z-, t)}{c(z-)}$.
	
	\item A \emph{boundary triangle} is a half-triangle $\righttri = \tri_+(0, 0, r)$ with base $B_+ = P_z \righttri = [0, r]$ where $B_+$ does not intersect $D(c)$. These are used to study the nonlinear Neumann condition $\frac{u_z}{c(0)} = (f(u_t))_t$.
	
	\item A \emph{plain triangle} is a triangle $\tri = \tri(z_0, 0, r)$ with base $B = P_z \tri \subseteq (0, \infty)$ not intersecting $D(c)$. These are used to cover the remaining space.
\end{itemize}

\begin{lemma} \label{lem:wellposed:plain}
	Let $\tri$ be a plain triangle with base $B$. Assume $r \norm{\frac{c_z}{c}}_\infty < 1$. Then \eqref{eq:IP_z} has a unique $C^1$-solution $u$ on $\tri$ and there exists a constant $C = C(r, \norm{\frac{c_z}{c}}_\infty)$ such that the solution operator $\Phi \colon C^1(B) \times C(B) \to C^1(\tri), (u_0, u_1) \mapsto u$ satisfies $\norm{\Phi} \leq C$.
\end{lemma} 

\begin{proof}
	This follows immediately from \cref{cor:linear_uniqueness} and \cref{cor:operator_estimates}.
\end{proof}

\begin{lemma} \label{lem:wellposed:jump}
	Let $\tri$ be a jump triangle with base $B$. Assume $r \norm{\frac{c_z}{c}}_\infty < 1$. Then \eqref{eq:IP_z} has a unique $C^1$-solution $u$ on $\tri$ and there exists a constant $C = C(r, \norm{\frac{c_z}{c}}_\infty)$ such that the solution operator $\Phi \colon \spacex{B} \times C(B) \to \spacext{\tri}, (u_0, u_1) \mapsto u$ satisfies $\norm{\Phi} \leq C$.
\end{lemma} 

\begin{proof}
	Let $\tri = \tri(z_0, 0, r)$. 
	If $u \colon \Delta \to \R$ is a solution of \eqref{eq:IP_z}, then by defining $b \colon [0, r] \to \R, b(t) = u(z_0, t)$ and using \cref{lem:dirichlet_phi} we have 
	\begin{align}\label{eq:loc:jump1}
		u(z, t) = \begin{cases}
			\Phi_+(b, u_0, u_1)(z, t), &z \geq z_0, \\
			\Phi_-(b, u_0, u_1)(z, t), &z \leq z_0.
		\end{cases}
	\end{align}
	On the other hand, if $b \in C^1([0, r])$ with $b(0) = u_0(z_0)$ and $b'(0) = u_1(z_0)$ is given, then the function $u$ defined by \eqref{eq:loc:jump1} satisfies $u, u_t \in C(\Delta)$ as $\Phi_\pm(b, u_0, u_1)$ and $\Phi_\pm(b, u_0, u_1)_t$ coincide with $b$ resp. $b'$ at the boundary $z=z_0$. Hence, $u$ solves \eqref{eq:IP_z} if and only if $u_x$ is continuous, i.e. 
	\begin{align}\label{eq:loc:jump2}
		\frac{u_z(z_0+, t)}{c(z_0+)} = \frac{u_z(z_0-, t)}{c(z_0-)}
	\end{align}
	holds for all $t \in [0, r]$. Using \eqref{eq:loc:jump1}, we can write \eqref{eq:loc:jump2} as
	\begin{align*}
		\frac{1}{c(z_0-)} \Phi_-(b, u_0, u_1)_z(z_0, t) 
		= \frac{1}{c(z_0+)} \Phi_+(b, u_0, u_1)_z(z_0, t)
	\end{align*}
	or as
	\begin{align*}
		b'(t) 
		= \gamma \left(\frac{1}{c(z_0-)} \left(b'(t) - \Phi_-(b, u_0, u_1)_z(z_0, t)\right) 
		+ \frac{1}{c(z_0+)} \left({b'(t) + \Phi_+(b, u_0, u_1)_z(z_0, t)}\right)\right)
	\end{align*}
	with
	\begin{align*}
		\gamma \coloneqq \left(\frac{1}{c(z_0-)} + \frac{1}{c(z_0+)}\right)^{-1}
	\end{align*}
	We denote the right-hand side by $T(b)(t)$ and show now that $\Psi \colon b \mapsto u_0(z_0) + \int_{0}^{(\impvar)} T(b)(\tau) \der \tau$ is a strict contraction in the space $X \coloneqq \{b \in C^1([0, r]) \mid b(0) = u_0(z_0)\}$ with norm $\norm{b}_X = \sup \set{\ee^{- \mu t}\abs{b'(t)} \colon t \in [0, r]}$, where $\mu > 0$ will be chosen later. So let $b, \tilde b \in X$ and write $\hat b \coloneqq b - \tilde b$. Next we estimate
	\begin{align*}
		&\abs{\Psi(b)(t) - \Psi(\tilde b)(t)} \\
		&\quad= \gamma \abs{\frac{1}{c(z_0-)} \left(\hat b'(t) - \Phi_-(\hat b, 0, 0)_z(z_0, t)\right) 
			+ \frac{1}{c(z_0+)} \left(\hat b'(t) + \Phi_+(\hat b, 0, 0)_z(z_0, t)\right)} \\
		&\quad\leq \gamma \left(\frac{1}{c(z_0-)} \beta \int_{0}^{t} \abs{\hat b'(\tau)} \der \tau + \frac{1}{c(z_0+)} \beta \int_{0}^{t} \abs{\hat b'(\tau)} \der \tau\right) \\ 
		&\quad= \beta \int_{0}^{t} \abs{\hat b'(\tau)} \der \tau
		\leq \beta \norm{\hat b}_X \int_{0}^{t} \ee^{\mu\tau} \der \tau
		\leq \frac{\beta}{\mu} \ee^{\mu t} \norm{\hat b}_X,
	\end{align*}
	where $\beta$ is the constant from \cref{lem:fine_boundary_estimate}. If we choose $\mu > \beta$, then $\Psi$ is a strict contraction so that $b = \Psi(b)$ has a unique solution by Banach's fixed-point theorem. Using \cref{rem:estimate:boundary_solution_op}, the fixed-point theorem also shows that $b$ linearly and continuously depends on $u_0$ and $u_1$. Moreover, boundedness of the linear solution operator $\Phi$ then follows from \eqref{eq:loc:jump1}.
\end{proof}

\begin{lemma} \label{lem:wellposed:boundary}
	Let $\righttri$ be a boundary triangle with base $B_+$. Assume $r \norm{\frac{c_z}{c}}_\infty < 1$. Then \eqref{eq:IP_z} has a unique $C^1$-solution on $\righttri$.
\end{lemma} 

\begin{proof}
	As in the previous lemma, we write $b(t) = u(0, t)$, Then $u$ is a solution on $\righttri$ if and only if $u = \Phi_+(b, u_0, u_1)$ and
	\begin{align*}
		\dv{f(u_t(0, t))}{t} = 
		\frac{u_z(0, t)}{c(0)}.
	\end{align*}
	We may rewrite the latter equation as
	\begin{align*}
		\dv{f(b'(t))}{t} =
		\frac{1}{c(0)}\Phi_+(b, u_0, u_1)_z(0, t).
	\end{align*}
	Replacing $b(t)$ with $d(t) \coloneqq f(b'(t))$, where $b$ can be reconstructed from $d$ via $b_d(t) \coloneqq u_0(0) + \int_{0}^{t} f^{-1}(d(\tau)) \der \tau$ we are left with solving
	\begin{align}\label{eq:loc:boundary_equation}
		d'(t) = \frac{1}{c(0)} \Phi_+(b_d, u_0, u_1)_z(0, t).
	\end{align}
	Therefore, it suffices to show that \eqref{eq:loc:boundary_equation} with initial datum $d(0) = f(u_1(0))$ has a unique solution.
	\medskip 

	\textbf{Uniqueness:}
	Assume that $d, \tilde d$ are solutions to \eqref{eq:loc:boundary_equation} that coincide up to time $t_\star \geq 0$, but not at time $t_n$ for some $t_n \geq 0$ with $t_n \downarrow t_\star$ as $n \to \infty$. Define $\delta(t) \coloneqq \abs{f^{-1}(d(t)) - f^{-1}(\tilde d(t))}$. For $\varepsilon > 0$ consider the function
	\begin{align*}
		h_\varepsilon(t) \coloneqq \varepsilon(1 + t - t_\star) + \frac{1}{c(0)} \int_{t_\star}^{t} \left(- \delta(s) + \beta \int_{t_\star}^{s} \delta(\tau) \der \tau\right) \der s,
	\end{align*}
	where $\beta$ is the constant from \cref{lem:fine_boundary_estimate}.

	\medskip
	\textit{Claim:} The inequality $\abs{d(t) - \tilde d(t)} < h_\varepsilon(t)$ holds for all $t \geq t_\star$.
	
	Clearly, the claim holds true for $t = t_\star$, and thus by continuity for $t$ close to $t_\star$. Assume the claim is false. Then there exists some minimal $t_i > t_\star$ such that $\abs{d(t_i) - \tilde d(t_i)} = h_\varepsilon(t_i)$. W.l.o.g. assume that $d(t_i) \geq \tilde d(t_i)$. Since $d(t) - \tilde d(t) < h_\varepsilon(t)$ for $t_\star\leq t < t_i$, we get $d'(t_i) - \tilde d'(t_i) \geq h_\varepsilon'(t_i)$ which implies
	\begin{align*}
		\frac{1}{c(0)} \Phi_+(b_d, 0, 0)_z(0, t_i) - \frac{1}{c(0)} \Phi_+(b_{\tilde d}, 0, 0)_z(t_i) \geq \varepsilon + \frac{1}{c(0)} \left(- \delta(t_i) + \beta \int_{t_\star}^{t_i} \delta(\tau) \der \tau\right) 	
	\end{align*}
    and hence 
    \begin{equation} \label{eq:loc:for_contra}
    \Phi_+(b_d - b_{\tilde d}, 0, 0)_z(0, t_i) + \delta(t_i) > \beta \int_{t_\star}^{t_i} \delta(\tau) \der \tau \geq 0.
	\end{equation}
	On the other hand, setting $b \coloneqq b_d - b_{\tilde d}$ we have
	\begin{align*}
		\abs{\Phi_+(b, 0, 0)_z(0, t_i) + b'(t_i)} \leq \beta \int_{t_\star}^{t_i} \abs{b'(\tau)} \der \tau
	\end{align*}
	due to \cref{lem:fine_boundary_estimate}.
	Since $b'(t_i) = f^{-1}(d(t_i)) - f^{-1}(\tilde d(t_i))$ and since $f^{-1}$ is increasing, we see that $b'(t_i) = \delta(t_i)$. Combining these facts, we find
	\begin{align*}
		\abs{\Phi_+(b, 0, 0)_z(0, t_i) + \delta(t_i)} \leq \beta \int_{t_\star}^{t_i} \delta(\tau) \der \tau
	\end{align*}
	which contradicts \eqref{eq:loc:for_contra}. So the claim holds. 
	
	Letting $\varepsilon$ go to $0$, we obtain
	\begin{align*}
		\abs{d(t) - \tilde d(t)} \leq \frac{1}{c(0)} \int_{t_\star}^{t} \left(- \delta(s) + \beta \int_{t_\star}^{s} \delta(\tau) \der \tau\right) \der s
	\end{align*}
	for any $t \geq t_\star$. Fubini implies that the term on the right-hand side is negative for $t \in (t_\star, t_\star + \frac{1}{\beta})$, a contradiction.
	\medskip
	
	\textbf{Existence:}
	Let $D, \mu > 0$. Consider the set
	\begin{align*}
		K \coloneqq \{d \in W^{1, \infty}([0, r]) \colon d(t_0) = f^{-1}(u_1(0)), \abs{d(t)} \leq D \ee^{\mu t}, \abs{d'(t)} \leq D \mu \ee^{\mu t} \mbox{ for } t\in [0,r]\},
	\end{align*}
	which is a convex and compact subset of $C([0, r])$, as well as the operator
	\begin{align*}
		T \colon K \to C([0, r]), \quad T(d)(t) = f^{-1}(u_1(0)) + \frac{1}{c(0)} \int_{t_0}^t \Phi_+(b_d, u_0, u_1)_z(0, \tau) \der \tau.
	\end{align*}
	We choose $D \coloneqq \max\set{\abs{f^{-1}(u_1(0))}, 1}$, so that $K$ is nonempty as it contains the constant function $d \equiv f^{-1}(u_1(0))$. 
	To see that $T$ is continuous, let $d_n \in K$ with $d_n \to d$ in $C([0, r])$ as $n \to \infty$. As $f^{-1}$ is uniformly continuous on $[-D \ee^{\mu r}, D \ee^{\mu r}]$, we have $f^{-1} \circ d_n \to f^{-1} \circ d$ in $C([0, r])$, from which it follows that
	\begin{align*}
		b_{d_n} &= u_0(0) + \int_0^{(\impvar)} f^{-1}(d_n(\tau)) \der \tau
	\intertext{converges to}
		b_{d} &= u_0(0) + \int_0^{(\impvar)} f^{-1}(d(\tau)) \der \tau.
	\end{align*}
	in $C^1([0, r])$. Due to \cref{rem:estimate:boundary_solution_op}, the operator $\Phi_+(\impvar, u_0, u_1) \colon C^1([0, r]) \to C^1(\tri_+)$ is continuous. Hence $T(d_n) \to T(d)$ in $C([0, r])$ as $n \to \infty$.
	
	To check that $T$ maps into $K$, we need to verify that for any $d \in K$ one has
	\begin{align}\label{eq:loc:exist:1}
		\abs{T(d)'(t)} \leq D \mu \ee^{\mu t}.
	\end{align}
	Notice that $\abs{d(t)} \leq D \ee^{\mu t}$ follows from \eqref{eq:loc:exist:1} by integration.
	By assumption \eqref{ass:linear_growth} on the growth on $f$ we have $\abs{f^{-1}(y)} \leq \frac{\abs{y} + B}{A}$, and in particular $\abs{b_d'(t)} = \abs{f^{-1}(d(t))} \leq  \frac{D \ee^{\mu t} + B}{A}$. We use this inequality, $\abs{b_d(t)} \leq \abs{u_0(0)} + t \norm{b_d'}_\infty$ as well as \cref{rem:estimate:boundary_solution_op} to estimate 
	\begin{align*}
		\abs{T(d)'(t)} 
		&= \frac{1}{c(0)} \abs{\Phi_+(b_d, u_0, u_1)_z(0, t)} \\
		&\leq \frac{C}{c(0)} \left( \norm{b_d}_{[0, t], C^1} + \norm{u_0}_{C^1} + \norm{u_1}_{\infty} \right) \\
		&\leq \frac{C}{c(0)} \left( (1 + t) \norm{b_d'}_{[0, t], \infty} + 2 \norm{u_0}_{C^1} + \norm{u_1}_{\infty} \right) \\
		&\leq \frac{C}{c(0)} \left( (1 + t) \frac{D \ee^{\mu t} + B}{A} + 2 \norm{u_0}_{C^1} + \norm{u_1}_{\infty} \right) \\
		&\leq \frac{C}{c(0)} \left( (1 + r) \frac{D + B}{A} + 2 \norm{u_0}_{C^1} + \norm{u_1}_{\infty} \right) \ee^{\mu t}.
	\end{align*}
	Therefore $T$ maps $K$ into itself if we choose
	\begin{align*}
		\mu \coloneqq \frac{C}{c(0) D} \left( (1 + r) \frac{D + B}{A} + 2 \norm{u_0}_{C^1} + \norm{u_1}_{\infty} \right).
	\end{align*}
    Hence existence follows by applying Schauder's fixed-point Theorem.
\end{proof}

	With these auxiliary results finished, we are able to prove the main theorem. 

\begin{proof}[Proof of \cref{thm:main} with additional assumption \eqref{ass:linear_growth}]~\\
	\textbf{Step 1 - Constructing a solution:}
	
	Denote by $\calC$ the set containing all jump, boundary and plain triangles where the heights $r$ have to satisfy $r \norm{\frac{c_z}{c}}_\infty < 1$.
	As we have just shown in the previous three lemmata, \eqref{eq:IP_z} admits a unique solution on each $\tri \in \calC$. 
	Since $\calC$ is closed with respect to finite intersection, we obtain a solution $u$ of \eqref{eq:IP_z} on $\cup_{\tri \in \calC} \tri$. 
	Note that $[0, \infty) \times [0, h) \subseteq \cup_{\tri \in T} \tri$ where 
	\begin{align*}
		h \coloneqq \tfrac12 \min\left\{\norm{\frac{c_z}{c}}_\infty^{-1}, \abs{d_1 - d_2} \colon d_1, d_2 \in D(c) \cup \set{0}, d_1 \neq d_2\right\}.
	\end{align*}
	By restriction, we therefore obtain a solution $u^{(1)}$ of \eqref{eq:IP} on $[0, \infty) \times [0, \tilde h]$ for any $0 < \tilde h < h$. Restarting with initial data $u_0^{(2)}(z) = u^{(1)}(z, \tilde h)$ and $u_1^{(2)}(z) = u_t^{(1)}(z, \tilde h)$, the above method yields a solution $u^{(2)}$ on 
	$[0, \infty) \times [0, \tilde h]$. We repeat this argument to construct solutions $u^{(k)}$ for $k \in \N$. Finally, we define the map $u \colon [0, \infty) \times [0, \infty) \to \R$ by $u(z, (k-1) \tilde h + \tau) = u^{(k)}(z, \tau)$ for $\tau \in [0, \tilde h]$, which solves \eqref{eq:IP}.
	\medskip
	
	\textbf{Step 2 - Uniqueness:}
	
	Assume that $u, \tilde u \colon \Omega \to \R$ are two different solutions to \eqref{eq:IP_z}, where $\Omega = \{(z, t) \mid t \leq h(z)\}$ is an admissible domain. 
	So there exists $(z_0, t_0) \in \Omega$ with $u(z_0, t_0) \neq \tilde u(z_0, t_0)$. 
	Consider the (possibly cut-off) triangle $\tri \coloneqq \tri(z_0, 0, t_0) \cap \{z \geq 0\}$ and define the set $N \coloneqq \{(z, t) \in \tri \mid u(z, t) \neq \tilde u(z, t)\}$ and $t_{\inf} \coloneqq \inf P_t(N)$, where $P_t$ denotes the projection onto the second variable. 
	Choose some sequence $(z_n, t_n) \in N$ with $t_n \to t_{\inf}$ and $z_n \to z_\infty \in [0, \infty)$.
	
	For $\varepsilon>0$ consider the (possibly cut-off) triangle $\tri_\varepsilon \coloneqq \tri\cap \tri(z_\infty, t_{\inf}, \varepsilon) \cap \{z \geq 0\}$ with base $B_\varepsilon$. 
	\medskip
	
	\textit{Claim:} $u(z, t_{\inf}) = \tilde u(z, t_{\inf})$ and $u_t(z, t_{\inf}) = \tilde u_t(z, t_{\inf})$ hold for all $z \in B_\varepsilon$.
	
	If $t_{\inf} = 0$, this holds because both $u$ and $\tilde u$ satisfy the same initial conditions. If $t_{\inf} > 0$, by assumption we have $u(z, t) = \tilde u(z, t)$ for $z \in B_\varepsilon$ and $t < t_{\inf}$ as $(z, t) \in \Delta$ and therefore also $u_t(z, t) = \tilde u_t(z, t)$, so that the claim is obtained by taking the limit $t \to t_{\inf}$.
	
	If we choose $\varepsilon$ small enough, then $\tri_\varepsilon$ is a jump (if $z_\infty \in D(c)$), boundary (if $z_\infty = 0$) or plain triangle (otherwise). By the previously established uniqueness results on these triangles, $u$ and $\tilde u$ must coincide on $\tri_\varepsilon$. But since $t_n \geq t_{\inf}$ for all $n$, we have $(z_n, t_n) \in \tri_\varepsilon$ for $n$ sufficiently large, so that $u(z_n, t_n) = \tilde u(z_n, t_n)$. This cannot be since $(z_n, t_n) \in N$.
\end{proof}

\begin{remark}[Modifications for the bounded domain version]
	 In order to capture the homogeneous Dirichlet boundary condition for the bounded domain version of the theorem, we also need to consider ''Dirichlet'' triangles $\lefttri$ with center $z_0 = L$. Problem \eqref{eq:IP} is well-defined on the domain $\lefttri$ assuming $r \norm{\frac{c_z}{c}}_\infty <1$. In fact the solution on ''Dirichlet'' triangles is simply given by $u = \Phi_-(0, u_0, u_1)$. We can then proceed as in the above proof to show existence and uniqueness of solutions.
\end{remark}

\section{Energy, Momentum, and Completion of \cref{thm:main}} \label{sec:energy}
We recall that the energy of \eqref{eq:IP} is given by
\begin{align*} 
	E(u, t) &\coloneqq \tfrac12 \int_{0}^{\infty} \left(V(x) u_t(x, t)^2 + u_x(x, t)^2\right) \der x + F(u_t(0, t)) \\
	&= \tfrac12 \int_0^\infty \left(\frac{1}{c(z)^2} u_t(z, t)^2 + \left(\frac{u_z(z, t)}{c(z)}\right)^2 \right) \cdot c(z) \der z + F(u_t(0, t)) \\
	&= \tfrac12 \int_0^\infty \frac{1}{c(z)} \left(u_t(z, t)^2 + u_z(z, t)^2 \right) \der z + F(u_t(0, t))
\end{align*}
where $F(y)=y f(y) - \int_{0}^{y} f(v) \der v$. In $(z, t)$--coordinates the momentum reads
\begin{align*}
	M(u, t) = \int_0^\infty \frac{1}{c} u_t \der z + f(u_t(0, t)).
\end{align*} 
We now show that both quantities are time-invariant.

\begin{proof}[Proof of \cref{thm:conservation}]
	Let $\Omega \subseteq [0, \infty) \times [0, \infty)$ be a Lipschitz domain such that $c$ is $C^1$ on $\Omega$. Recall that $(\partial_t \mp \partial_z) (u_t \pm u_z) u + \frac{c_z}{c} u_z = 0$. In the following, for a term $a(\pm, \mp)$ which may have $\pm$ or $\mp$ signs, we write ${\sum}^\pm a(\pm, \mp) \coloneqq a(+, -) + a(-, +)$.

	\textbf{Part 1: Energy.} With $\nu$ being the outer normal at $\bd \Omega$ we calculate
	\begin{align*}
		0 &= {\sum}^\pm \int_\Omega \left[ (\partial_t \mp \partial_z) (u_t \pm u_z) u + \frac{c_z}{c} u_z \right] \cdot \frac{1}{c} (u_t \pm u_z) \der (z, t)
		\\ &= {\sum}^\pm  \int_{\bd \Omega} (\nu_2 \mp \nu_1) \frac{1}{c} (u_t \pm u_z)^2 \der \sigma
		\\ &\quad+ {\sum}^\pm \int_\Omega \left(\frac{c_z}{c^2} u_z (u_t \pm u_z) - \frac{1}{c} (u_t \pm u_z) \cdot (\partial_t \mp \partial_z) (u_t \pm u_z) \mp \frac{c_z}{c^2} (u_t \pm u_z)^2\right) \der (z, t).
	\end{align*}
	The sum $\sum^\pm$ over the boundary integrals can be simplified to
	\begin{align*}
		{\sum}^\pm  \int_{\bd \Omega} (\nu_2 \mp \nu_1) \frac{1}{c} (u_t \pm u_z)^2 \der \sigma
		= \int_{\bd \Omega} \left(\frac{2}{c} \nu_2 (u_t^2 + u_z^2) - \frac{4}{c} \nu_1 u_t u_z\right) \der \sigma.
	\end{align*}
	The sum $\sum^\pm$ of the integrands in the integral over $\Omega$ vanishes as can be seen by the following calculation using once more the differential equation $(\partial_t \mp \partial_z) (u_t \pm u_z) u + \frac{c_z}{c} u_z = 0$:
	\begin{align*}
		&{\sum}^\pm \left(\frac{c_z}{c^2} u_z (u_t \pm u_z) - \frac{1}{c} (u_t \pm u_z) \cdot (\partial_t \mp \partial_z) (u_t \pm u_z) \mp \frac{c_z}{c^2} (u_t \pm u_z)^2\right)
		\\ &\quad= {\sum}^\pm \left(\frac{c_z}{c^2} u_z (u_t \pm u_z) + \frac{1}{c} (u_t \pm u_z) \frac{c_z}{c} u_z \mp \frac{c_z}{c^2} (u_t \pm u_z)^2\right)
		\\ &\quad= \frac{c_z}{c^2} {\sum}^\pm \left(2 u_z (u_t \pm u_z) \mp (u_t \pm u_z)^2\right) = 0.
	\end{align*}
	Hence
	\begin{align}\label{eq:loc:energy_alt:1}
		\int_{\bd \Omega} \left(\frac{2}{c} \nu_2 (u_t^2 + u_z^2) - \frac{4}{c} \nu_1 u_t u_z\right) \der \sigma = 0.
	\end{align}

	Since $D(c)$ and $D(c_z)$ are discrete sets, we find an increasing sequence $0 = a_1 < a_2 < a_3 < \dots$ with $a_k \to \infty$ as $k \to \infty$ such that $D(c) \cup D(c_z) \subseteq \set{a_k \colon k \in \N}$. 
	
	Now let $t_1 < t_2 \in \R$ and $K \in \N$. We choose $\Omega = [a_k, a_{k+1}] \times [t_1, t_2]$ and sum \eqref{eq:loc:energy_alt:1} from $k=1$ to $K$. As terms along common boundaries cancel, we obtain
	\begin{align*}
		0 &= \int_{\bd ([0, a_{K+1}] \times [t_1, t_2])} \left(\frac{2}{c} \nu_2 (u_t^2 + u_z^2) - \frac{4}{c} \nu_1 u_t u_z\right) \der \sigma
	\end{align*}
	or equivalently
	\begin{multline*}
		\tfrac12 \left.\int_{0}^{a_{K+1}} \left(\frac{1}{c} u_t^2 + \frac{1}{c} u_z^2\right) \der z \right\vert _{t=t_2}
		\\
		=\tfrac12 \left.\int_{0}^{a_{K+1}} \left(\frac{1}{c} u_t^2 + \frac{1}{c} u_z^2\right) \der z \right\vert _{t=t_1}
		- \left. \int_{t_1}^{t_2} \frac{1}{c} u_t u_z \der t \right\vert _{z=a_{K+1}} + \left. \int_{t_1}^{t_2} \frac{1}{c} u_t u_z \der t \right\vert _{z=0}.
	\end{multline*}
    The estimates established in \cref{cor:operator_estimates} and the assumptions on the initial conditions $u_0, u_1$ show that $u_t(z, t)$ and $u_z(z, t)$ converge to $0$ as $z \to \infty$ uniformly on $[t_1, t_2]$. In the limit $K \to \infty$, we thus obtain
	\begin{align*}
		\tfrac12 \left.\int_{0}^{\infty} \left(\frac{1}{c} u_t^2 + \frac{1}{c} u_z^2\right) \der z \right\vert _{t=t_2}
		= \tfrac12 \left.\int_{0}^{\infty} \left(\frac{1}{c} u_t^2 + \frac{1}{c} u_z^2\right) \der z \right\vert _{t=t_1}
		+ \left. \int_{t_1}^{t_2} \frac{1}{c} u_t u_z \der t \right\vert _{z=0}.
	\end{align*}
	Switching back to $(x,t)$--coordinates, we infer
	\begin{align*}
		\left. \int_{t_1}^{t_2} u_t u_x \der t \right\vert _{x=0}
		&= \int_{t_1}^{t_2} u_t(0, t) u_x(0, t) \der t\\
		&= \int_{t_1}^{t_2} u_t(0, t) f(u_t(0, t))_t \der t
		= F(u_t(0, t_2)) - F(u_t(0, t_1))
	\end{align*}
	where the last equality is due to \cref{lem:aux:chain_rule_formula}. This shows the claimed energy conservation:
	\begin{align*}
		\tfrac12 \left.\int_{0}^{\infty} \left(V(x) u_t^2 + u_x^2\right) \der x + F(u_t(0, t)) \right\vert _{t=t_2}
		= \tfrac12 \left.\int_{0}^{\infty} \left(V(x) u_t^2 + u_x^2\right) \der x + F(u_t(0, t)) \right\vert _{t=t_1}.
	\end{align*}

	\textbf{Part 2: Momentum.}
	We calculate
	\begin{equation} \label{eq:loc:momentum_alt:1}
	\begin{split}
		0 &= {\sum}^\pm  \int_{\Omega} \frac{1}{c} \left[(\partial_t \pm \partial_z) (u_t \mp u_z) + \frac{c_z}{c} u_z \right] \der (z, t) \\ 
		&= {\sum}^\pm  \int_{\bd \Omega} (\nu_2 \pm \nu_1) \frac{1}{c} (u_t \mp u_z) \der \sigma \\
		&\quad +{\sum}^\pm \int_\Omega \left(\pm \frac{c_z}{c^2} (u_t \mp u_z) + \frac{c_z}{c^2} u_z\right) \der (z, t) \\
		&= 2 \int_{\bd \Omega} \left(\nu_2 \frac{1}{c} u_t - \nu_1 \frac{1}{c} u_z\right) \der \sigma. 
    \end{split}
	\end{equation}
    Again we choose $\Omega = [a_{k}, a_{k+1}] \times [t_1, t_2]$, and sum \eqref{eq:loc:momentum_alt:1} from $k=1$ to $K$. As before all terms along common boundaries cancel, whence we obtain
	\begin{align*}
		\left. \int_0^{a_{K+1}} \frac{1}{c} u_t \der z \right\vert _{t=t_2}
		= \left. \int_0^{a_{K+1}} \frac{1}{c} u_t \der z \right\vert _{t=t_1} 
		+ \left. \int_{t_1}^{t_2} \frac{1}{c} u_z \der t \right\vert _{z=a_{K+1}}
		- \left. \int_{t_1}^{t_2} \frac{1}{c} u_z \der t \right\vert _{z=0}.
	\end{align*}
	Since
	\begin{align*}
		\left. \int_{t_1}^{t_2} \frac{1}{c} u_z \der t \right\vert _{z=0}
		= \int_{t_1}^{t_2} f(u_t(0, t))_t \der t = f(u_t(0, t_2)) - f(u_t(0, t_1)),
	\end{align*}
	in the limit $K \to \infty$ we find the claimed momentum conservation:
	\begin{align*}
		\left. \int_0^\infty \frac{1}{c^2} u_t \der x + f(u_t(0, t)) \right\vert _{t=t_2}
		= \left. \int_0^\infty \frac{1}{c^2} u_t \der x + f(u_t(0, t)) \right\vert _{t=t_1}.
		&\qedhere
	\end{align*}
\end{proof}

In \cref{sec:main_proof}, we required an extra growth condition \eqref{ass:linear_growth} on $f$ in order to prove a first version of Theorem~\ref{thm:main}. We now discuss how to exploit the energy conservation to eliminate this extra growth assumption and prove Theorem~\ref{thm:main} in full generality.

\begin{lemma} \label{lem:localized_energy}
	For $t > 0$ the estimate
	\begin{align*}
		F(u_t(0, t)) 
		\leq F(u_1(0)) + \tfrac12 \int_{0}^{\kappa^{-1}(t)} \left(V(x) u_1(x)^2 + u_{0,x}(x)^2\right) \der x
	\end{align*}
	holds, where $\kappa(x)=\int_0^x \tfrac{1}{c(s)} \der s = \int_0^x \sqrt{V(s)}\der s$. 
\end{lemma}

\begin{proof}
	Fix $t_1 > 0$, let $\varepsilon > 0$ and define modified initial data $\tilde u_0, \tilde u_1 \colon [0, \infty) \to \R$ by setting
	\begin{align*}
		\tilde u_0'(z) = \begin{cases}
			u_0'(z), & z \leq t_1, \\
			\frac{t_1 + \varepsilon - z}{\varepsilon} u_0'(t_1), & t_1 \leq z \leq t_1 + \varepsilon, \\
			0, & z \geq t_1 + \varepsilon, 
		\end{cases}\qquad
		\tilde u_1(z) = \begin{cases}
			u_1(z), & z \leq t_1, \\
			\frac{t_1 + \varepsilon - z}{\varepsilon} u_1(t_1), & t_1 \leq z \leq t_1 + \varepsilon, \\
			0, & z \geq t_1 + \varepsilon, 
		\end{cases}
	\end{align*}
	and $\tilde u_0(0) = u_0(0)$. Denote the solution to \eqref{eq:IP_z} corresponding to these initial data by $\tilde u$. By uniqueness of the solution, $u(z, t) = \tilde u(z, t)$ for $\abs{z} + \abs{t} \leq t_1$. In particular, $\tilde u_t(0, t_1) = u_t(0, t_1)$. This yields
	\begin{align*}
		&F(u_t(0, t_1)) \\
		&\quad= F(\tilde u_t(0, t_1)) \leq E(\tilde u, t_1) 
		= E(\tilde u, 0) \\
		&\quad= F(\tilde u_t(0, 0)) + \tfrac12 \int_{0}^{\infty} \left(V(x) \tilde u_1(x)^2 + \tilde u_0'(x)^2\right) \der x \\
		&\quad= F(u_1(0)) + \tfrac12 \int_{0}^{\kappa^{-1}(t_1)} \left(V(x) u_1(x)^2 + u_0'(x)^2\right) \der x
		+ \tfrac12 \int_{\kappa^{-1}(t_1)}^{\kappa^{-1}(t_1 + \varepsilon)} \left(V(x) \tilde u_1(x)^2 + \tilde u_0'(x)^2\right) \der x.
	\end{align*}
	Letting $\varepsilon \to 0$, the last term goes to $0$. 
\end{proof}

\begin{proof}[Proof of \cref{thm:main} without additional assumption \eqref{ass:linear_growth}]~\\
	Fix $T > 0$ and let
	\begin{align*}
		C \coloneqq F(u_1(0)) + \tfrac12 \int_{0}^{\kappa^{-1}(T)} \left(V(x) u_1(x)^2 + u_{0,x}(x)^2\right) \der x
	\end{align*}
	Since $F(y)=\int_0^y f(y)-f(x)\,dx$ we see that $F(y) \to \infty$ as $y \to \pm \infty$. Therefore the set $\{y \colon F(y) \leq C\}$ is contained in the interval $[-K, K]$ for some $K > 0$. 
	Now consider the cut-off version of $f$ given by
	\begin{align*}
		f_K(y) = \begin{cases}
			y - K + f(K), & y \geq K, \\
			f(y), & -K \leq y \leq K, \\
			y + K + f(-K), & y \leq -K, \\
		\end{cases}
	\end{align*}
	which satisfies the growth conditions from \cref{sec:main_proof}. Therefore, \cref{thm:main} can be applied to \eqref{eq:IP} with $f$ replaced by $f_K$ and we obtain a solution $u_K$ on $[0,\infty)\times [0,T]$. \cref{lem:localized_energy} gives $F_K(u_{K,t}(0,t))\leq C$, so that $u_{K,t}(0,t)$ takes values in $[-K,K]$ where the functions $f, F$ and $f_k, F_k$ coincide. Hence $u_K$ solves the original problem \eqref{eq:IP} up to time $T$.
\end{proof}

Next, we verify that $C^1$-solutions to \eqref{eq:IP} are indeed weak solutions in the sense of \cref{def:weak_solution}.
\begin{proposition}\label{prop:c1_implies_weak}
	A $C^1$-solution to \eqref{eq:IP} is also a weak solution to \eqref{eq:IP}.
\end{proposition}
\begin{proof}
	Let $u$ be a $C^1$-solution to \eqref{eq:IP}. We have to show that
	\begin{align*}
		0 &= \int_0^\infty \int_0^\infty \left(V(x) u_t \varphi_t - u_x \varphi_x\right) \der x \der t
		+ \int_0^\infty f(u_t(0, t)) \varphi_t(0, t) \der t
		\\ &\quad+\int_{0}^{\infty} V(x) u_1(x) \varphi(x, 0) \der x
		+ f(u_1(0)) \varphi(0, 0)
	\end{align*}
	holds for all $\varphi \in C_c^\infty([0, \infty) \times [0, \infty))$. 
	
	Let $\Omega \subseteq [0, \infty) \times [0, \infty)$ be a Lipschitz domain such that $c$ is $C^1$ on $\Omega$. Denoting the outer normal at $\bd \Omega$ by $\nu$, we obtain
	\begin{align*}
		0 &= \int_{\Omega} \left[ (\partial_t - \partial_z) (u_t + u_z) + \frac{c_z}{c} u_z \right] \cdot \frac{1}{c} \varphi \der (z, t)
		\\ &= \int_{\bd \Omega} \frac{1}{c} (u_t + u_z) \varphi \cdot (\nu_2 - \nu_1) \der \sigma + 
		\int_\Omega \left(\frac{c_z}{c^2} u_z \varphi - (u_t + u_z) (\partial_t - \partial_z) \left[\frac{1}{c} \varphi \right]\right) \der (z, t)
		\\ &= \int_{\bd \Omega} \left(\frac{1}{c} u_t \varphi \nu_2 - \frac{1}{c} u_z \varphi \nu_1\right) \der \sigma + \int_{\Omega} \left(\frac{1}{c} u_z \varphi_z - \frac{1}{c} u_t \varphi_t\right) \der (z, t)
		\\ &\quad+\int_{\bd \Omega} \left(\frac{1}{c} u_z \varphi \nu_2 - \frac{1}{c} u_t \varphi \nu_1\right) \der \sigma + \int_{\Omega} \left(u_t \partial_z \left[\frac{1}{c} \varphi \right] - u_z \partial_t \left[\frac{1}{c} \varphi \right]\right) \der (z, t).
	\end{align*}
	We next show that the sum of the last two integrals equals zero. First, we calculate
	\begin{eqnarray*}
    \lefteqn{\int_{\bd \Omega} \left(\frac{1}{c} u_z \varphi \nu_2 - \frac{1}{c} u_t \varphi \nu_1\right) \der \sigma + \int_{\Omega} \left(u_t \partial_z \left[\frac{1}{c} \varphi \right] - u_z \partial_t \left[\frac{1}{c} \varphi \right]\right) \der (z, t)}
		\\ &=&   \int_{\bd \Omega} \left(\frac{1}{c} u_z \varphi \nu_2 - \frac{1}{c} u_t \varphi \nu_1 + u \partial_z \left[ \frac{1}{c} \varphi \right] \nu_2 - u \partial_t \left[ \frac{1}{c} \varphi \right] \nu_1\right) \der \sigma
		\\ &=& \int_{\bd \Omega} \left(\nu_2 \partial_z - \nu_1 \partial_t \right) \left[\frac{1}{c} u \varphi \right] \der \sigma.
	\end{eqnarray*}
	Let $\gamma \colon [0, l] \to \R$ be a positively oriented parametrization of $\bd \Omega$ by arc length. As $\nu$ is the outer normal at $\bd \Omega$, the identity $\gamma' = (\nu_2, - \nu_1)^\top$ holds. Hence,
	\begin{align*}
		\int_{\bd \Omega} \left(\nu_2 \partial_z - \nu_1 \partial_t \right) \left[\frac{1}{c} u \varphi \right] \der \sigma
		= \int_{\bd \Omega} \begin{pmatrix}\nu_2 \\ -\nu_1\end{pmatrix} \cdot \nabla \left[\frac{1}{c} u \varphi \right] \der \sigma
		= \int_{0}^{l} \gamma'(s) \cdot \nabla \left[\frac{1}{c} u \varphi \right](\gamma(s)) \der s
		= 0
	\end{align*}
	as $\gamma$ is closed. Thus we have shown
	\begin{align}\label{eq:loc:weak1}
		0 = \int_{\bd \Omega} \left(\frac{1}{c} u_t \varphi \nu_2 - \frac{1}{c} u_z \varphi \nu_1\right) \der \sigma + \int_{\Omega} \left(\frac{1}{c} u_z \varphi_z - \frac{1}{c} u_t \varphi_t\right) \der (z, t).
	\end{align}

	As in the proof of \cref{thm:conservation} we choose an increasing sequence $0 = a_1 < a_2 < a_3 < \dots$ with $a_k \to \infty$ as $k \to \infty$ such that $D(c) \cup D(c_z) \subseteq \set{a_k \colon k \in \N}$. We take $\Omega = [a_{k}, a_{k+1}] \times [n, n+1]$ in \eqref{eq:loc:weak1} and sum over $k \in \N$ and $n \in \N_0$. Using that boundary terms along common boundaries cancel out, the fact that $\varphi$ has compact support, and \eqref{eq:IP}, we obtain
	\begin{align*}
		0 &= \int_{\bd [0, \infty)^2} \left(\frac{1}{c} u_t \varphi \nu_2 - \frac{1}{c} u_z \varphi \nu_1\right) \der \sigma + \int_{[0, \infty)^2} \left(\frac{1}{c} u_z \varphi_z - \frac{1}{c} u_t \varphi_t\right) \der (z, t)
		\\ &= - \int_0^\infty \left[\frac{1}{c} u_t \varphi\right](z, 0) \der z+ \int_0^\infty \left[ \frac{1}{c} u_z \varphi \right](0, t) \der t + \int_0^\infty \int_0^\infty \left(\frac{1}{c} u_z \varphi_z - \frac{1}{c} u_t \varphi_t\right) \der z \der t
		\\ &= - \int_0^\infty V(x) u_t(x, 0) \varphi(x, 0) \der x + \int_0^\infty u_x(0, t) \varphi(0, t) \der t + \int_0^\infty \int_0^\infty \left(u_x \varphi_x - V(x) u_t \varphi_t\right) \der x \der t
		\\ &= - \int_0^\infty V(x) u_1(x) \varphi(x, 0) \der x + \int_0^\infty \left(f(u_t(0, t))\right)_t \varphi(0, t) \der t + \int_0^\infty \int_0^\infty \left(u_x \varphi_x - V(x) u_t \varphi_t\right) \der x \der t
		\\ &= - \int_0^\infty V(x) u_1(x) \varphi(x, 0) \der x - \int_0^\infty f(u_t(0, t)) \varphi_t(0, t) \der t - f(u_1(0)) \varphi(0, 0) 
		\\ &\qquad+ \int_0^\infty \int_0^\infty \left(u_x \varphi_x - V(x) u_t \varphi_t\right) \der x \der t
	\end{align*}
which finishes the proof.
\end{proof}

\section{Wellposedness} \label{sec:wellposedness}

The section completes the proof of the wellposedness claim stated in \cref{thm:main}. To be precise, \eqref{eq:IP} is wellposed in the following sense. The spaces $\spacext{[0,\infty)\times[0,T]}$, $\spacex{[0,\infty)}$, and $C([0,\infty))$ are endowed with uniform convergence on compact sets. 
\begin{proposition} \label{prop:wp}
	Assume that $u_0^{(n)}, u_1^{(n)}$ are initial data with $u_0^{(n)} \to u_0$ in $\spacex{[0, \infty)}$ and $u_1^{(n)} \to u_1$ in $C([0, \infty))$, and denote by $u^{(n)}$ and $u$ the solutions of \eqref{eq:IP_z} corresponding to these initial data. Then for any $T > 0$, we have $u^{(n)} \to u$ in $\spacext{[0, \infty) \times [0, T]}$.
\end{proposition}

\begin{proof}[Sketch of proof]
	We proceed similar to the proof of \cref{thm:main}. Choose some
	\begin{align*}
		0 < \bar r < \min\set{\left(5 - \sqrt{17}\right) \norm{\frac{c_z}{c}}_{\infty}^{-1}, \abs{z_1 - z_2} \colon z_1, z_2 \in D(c) \cup \set{0}, z_1 \neq z_2}.
	\end{align*}
	and let $\beta$ be as in \cref{lem:fine_boundary_estimate} with $r = \bar r$. The choice of $\bar r$ implies $\beta \bar r < \tfrac{4(5-\sqrt{17})}{(-3+\sqrt{17})(-1+\sqrt{17})} = 1$ as well as $q \coloneqq \bar r \norm{\frac{c_z}{c}}_\infty < 1$.

	Denote by $\calC$ the set containing all triangles $\tri$ that are of jump-type or plain-type and such that their base-radii $r$ are at most $\bar r$. Then by \cref{lem:wellposed:plain,lem:wellposed:jump}, there exists a constant $C > 0$ such that 
	\begin{align*}
		\norm{u^{(n)} - u}_{\spacext{\tri}} \leq C \max\set{\norm{u_0^{(n)} - u_0}_{\spacex{[0, \infty)}}, \norm{u_1^{(n)} - u_1}_{C([0, \infty))}}
	\end{align*}
	holds for each $\tri \in \calC$. 

	We also consider a single boundary-type triangle $\righttri$ with center $z_0 = 0$ and height $\bar r$. Writing $b(t) \coloneqq u(0, t)$, $b^{(n)}(t) \coloneqq u^{(n)}(0, t)$, $d(t) \coloneqq f(u_t(0, t))$ as well as $d^{(n)}(t) \coloneqq f(u^{(n)}_t(0, t))$, as in the proof of \cref{lem:wellposed:boundary} we obtain
	\begin{align*}
		d'(t) = \frac{1}{c(0)} \Phi_+(b, u_0, u_1)_z(0, t), \qquad
		\left(d^{(n)}\right)'(t) = \frac{1}{c(0)} \Phi_+(b^{(n)}, u_0^{(n)}, u_1^{(n)})_z(0, t).
	\end{align*}
	Setting $\hat b(t) \coloneqq u_0^{(n)}(0) - u_0(0) + t \left(u_1^{(n)}(0) - u_1(0)\right)$, we find
	\begin{align*}
		&c(0)\bigl(d'^{(n)}(t) - d'(t)\bigr) \\ 
		&\quad= \Phi_+(b^{(n)} - b, u_0^{(n)} - u_0, u_1^{(n)} - u_1)_z(0, t) \\
		&\quad= \Phi_+(\hat b, u_0^{(n)} - u_0, u_1^{(n)} - u_1)_z(0, t)+ \Phi_+(b^{(n)} - b - \hat b, 0, 0)_z(0, t) \\
		&\quad= \Phi_+(\hat b, u_0^{(n)} - u_0, u_1^{(n)} - u_1)_z(0, t)
		- \left[f^{-1}(d^{(n)}(t)) - f^{-1}(d(t)) - \left(u_1^{(n)}(0) - u_1(0)\right)\right] + \rho(n,t)
	\end{align*}	
    where \cref{lem:fine_boundary_estimate} gives
	\begin{align*}
		\abs{\rho(n,t)} \leq \beta \int_0^t \abs{f^{-1}(d^{(n)}(\tau)) - f^{-1}(d(\tau)) - u_1^{(n)}(0) + u_1(0)} \der \tau.	
	\end{align*}
	Multiplying with $\sign\left(d^{(n)}(t) - d(t)\right)$ and integrating, we obtain
	\begin{align*}
		&c(0)\abs{d^{(n)}(t) - d(t)} \\
		&\quad\leq c(0)\abs{d^{(n)}(0) - d(0)} \\
		&\qquad+ \int_0^t \left(\abs{\Phi_+(\hat b, u_0^{(n)} - u_0, u_1^{(n)} - u_1)_z(0, s)} - \abs{f^{-1}(d^{(n)}(s)) - f^{-1}(d(s))} + \abs{u_1^{(n)}(0) - u_1(0)}\right) \der s \\
		&\qquad+ \beta \int_0^t \int_0^s \abs{f^{-1}(d^{(n)}(\tau)) - f^{-1}(d(\tau)) - u_1^{(n)}(0) + u_1(0)} \der \tau \der s \\
		&\quad\leq \int_0^t \left(\abs{\Phi_+(\hat b, u_0^{(n)} - u_0, u_1^{(n)} - u_1)_z(0, s)} - \abs{f^{-1}(d^{(n)}(s)) - f^{-1}(d(s))} + \abs{u_1^{(n)}(0) - u_1(0)}\right) \der s \\
		&\qquad+ \beta \int_0^{\bar r} \int_0^t \left(\abs{f^{-1}(d^{(n)}(\tau)) - f^{-1}(d(\tau))} + \abs{u_1^{(n)}(0) - u_1(0)}\right) \der \tau \der s \\
		&\quad= \int_0^t \abs{\Phi_+(\hat b, u_0^{(n)} - u_0, u_1^{(n)} - u_1)_z(0, s)} \der s + (1 + \bar r \beta) t \abs{u_1^{(n)}(0) - u_1(0)} \\
		&\qquad- (1 - \bar r \beta) \int_0^t \abs{f^{-1}(d^{(n)}(s)) - f^{-1}(d(s))} \der s \\
		&\quad\leq \int_0^t \abs{\Phi_+(\hat b, u_0^{(n)} - u_0, u_1^{(n)} - u_1)_z(0, s)} \der s + (1 + \bar r \beta) t \abs{u_1^{(n)}(0) - u_1(0)} \\
		&\quad\leq\tilde C\left(\bar r, \norm{\frac{c_z}{c}}_\infty \right) \max\set{\norm{u_0^{(n)} - u_0}_{\spacex{[0, \infty)}}, \norm{u_1^{(n)} - u_1}_{C([0, \infty))}}.
	\end{align*}
	This shows the uniform convergence of $d^{(n)}$ to $d$ on $[0,\bar r]$ as $n\to\infty$. Since 
	$$
	b(t) = u_0(0) + \int_{0}^{t} f^{-1}(d(\tau)) \der \tau, \qquad b^{(n)}(t) = u_0(0) + \int_{0}^{t} f^{-1}(d^{(n)}(\tau)) \der \tau
	$$
	for $t \in [0, \bar r]$, it follows that $b^{(n)} \to b$ in $C^1([0, \bar r])$ as $n \to \infty$, and therefore we see that $u^{(n)} = \Phi_+(b^{(n)}, u_0^{(n)}, u_1^{(n)}) \to \Phi_+(b, u_0, u_1) = u$ in $C^1(\righttri)$.
	
	Combined, we find that that $u^{(n)} \to u$ in $\spacext{\mathcal{D}}$ where $\mathcal{D} \coloneqq \cup_{\tri \in \mathcal{C}} \tri$. Note that $[0, \infty) \times [0, \frac{\bar r}{2}] \subseteq \mathcal{D}$, so in particular $u^{(n)} \to u$ in $\spacext{[0, \infty) \times [0, \frac{\bar r}{2}]}$.
	Applying this result repeatedly $k$ times, we see that $u^{(n)} \to u$ in $\spacext{[0, \infty) \times [0, k \frac{\bar r}{2}]}$ where $k\in \N$ is chosen such that $k \frac{\bar r}{2} \geq T$.
\end{proof}

\section{Breather solutions and their regularity} \label{sec:regularity}

One can also consider \eqref{eq:IP} in the context of breather solutions, where a \emph{breather} is a time-periodic and spatially localized function. With time-period denoted by $T$, the time domain becomes the torus $\T \coloneqq \R /_T$ and after dropping the initial data, \eqref{eq:IP} reads 
\begin{align}\label{eq:wave}
	\begin{cases}
		V(x) u_{tt}(x, t) - u_{xx}(x, t) = 0, &x \in [0, \infty), t \in \T, \\
		u_x(0, t) = (f(u_t(0, t)))_t, &t \in \T.
	\end{cases}
\end{align}
In \cite{kohler_reichel} the case of a cubic boundary term $f(y) = \tfrac12 \gamma y^3$ ($\gamma \in \R \setminus \set{0}$) and a $2 \pi$-periodic step potential $V \colon \R \to \R$ given by
\begin{align}\label{ass:periodic_potential}\tag{A5}
	V(x) = \begin{cases}
		a, & \abs{x} < \pi \theta, \\
		b, & \theta \pi < \abs{x} < \pi,
	\end{cases}
\end{align}
where $b > a > 0$ and $\theta \in (0, 1)$ was discussed. It was shown that if $V$ satisfies
\begin{align}\label{ass:periodic_potential_cond}\tag{A6}
	4 \sqrt{a} \theta \omega \in 2 \N_0 + 1
	\quad\text{and}\quad
	4 \sqrt{b} \left(1 - \theta\right) \omega \in 2 \N_0 + 1,
\end{align}
where $\omega \coloneqq \frac{2 \pi}{T}$ is the frequency, then there exist infinitely many weak breather solutions $u$ of \eqref{eq:wave} with time-period $T$. A weak solution of \eqref{eq:wave} is defined next.

\begin{definition}\label{def:weak_periodic_solution}
	Let $f:\R\to\R$ be an increasing, odd homeomorphism. A \emph{weak solution} of \eqref{eq:wave} is a function $u \in H^1([0, \infty) \times \T)$ with $u(0,\cdot)\in W^{1,1}(\T)$ and $f(u_t(0, \cdot)) \in L^1(\T)$ which satisfies 
	\begin{align*}
		\int_{[0, \infty) \times \T} -V(x) u_t \varphi_t + u_x \varphi_x \der (x, t) - \int_\T f(u_t(0, t)) \varphi_t(0, t) \der t = 0
	\end{align*}
	for all test functions $\varphi \in C_c^\infty([0, \infty) \times \T)$.
\end{definition}

\begin{remark} 
	We require that the trace $u(0,\cdot)$ of $u$ at $x=0$ has an integrable weak first-order time derivative in order to give a pointwise meaning to $u_t(0,t)$ and, in particular, to define $f(u_t(0,t))$ pointwise almost everywhere. 
\end{remark}

In the setting of \cite{kohler_reichel} where $f(y) = \tfrac12 \gamma y^3$, one requires $u_t(0, t) \in L^3(\T)$ and
\begin{align*}
	2 \int_{[0, \infty) \times \T} -V(x) u_t \varphi_t + u_x \varphi_x \der (x, t) - \gamma \int_\T u_t(0, t)^3 \varphi_t(0, t) \der t = 0.
\end{align*}

In \cite[Theorem 4]{kohler_reichel} it was furthermore shown that weak solutions to \eqref{eq:wave} constructed in \cite{kohler_reichel} lie in $H^{\tfrac54 - \varepsilon}(\T, L^2(0, \infty)) \cap H^{\tfrac14 - \varepsilon}(\T, H^1(0, \infty))$ for $\varepsilon > 0$. 
Here, the Bochner spaces $H^{s}(\T, X)$ are defined by
\begin{align*}
	\norm{u}_{H^s(\T, X)}^2 \coloneqq \sum_{k \in \Z} (1 + k^2)^s \norm{\hat u_k}_X^2.
\end{align*}

In this section, we will show the following improved regularity result for breather solutions of \eqref{eq:wave}:

\begin{theorem}\label{thm:regularity}
	Assume \eqref{ass:nonlinearity}, \eqref{ass:periodic_potential}, \eqref{ass:periodic_potential_cond} that $f^{-1}$ is $r$-Hölder continuous with $r \in (0, 1)$ and that $u$ is a weak solution to \eqref{eq:wave}. Then $u$ is $\frac{T}{2}$-antiperiodic, lies in $C^{1,r}([0, \infty) \times \T)$ and is a $C^1$-solution to \eqref{eq:IP} with its own initial data, i.e. $u_0(x) = u(x, 0)$ and $u_1(x) = u_t(x, 0)$. In addition, there exists $C > 0$ such that $\abs{u(x, t)} \leq C \ee^{- \rho x}$ where $\rho \coloneqq \frac{\log(b) - \log(a)}{4 \pi}$.
\end{theorem}
Note that in the setting of \cite{kohler_reichel}, the assumptions of \cref{thm:regularity} are satisfied with $r = \tfrac13$. In the following, we are going to prove \cref{thm:regularity} and we will always assume the assumptions of \cref{thm:regularity}.

\subsection{Fourier decomposition of $V(x) \partial_t^2 - \partial_x^2$}

We denote by $e_k(t) \coloneqq \frac{1}{\sqrt{T}} \ee^{\ii k \omega t}$ the orthonormal Fourier base of $L^2(\T)$ and decompose $u$ in its Fourier series with respect to $t$:
\begin{align*}
	u(x, t) &= \sum_{k \in \Z} \hat u_k(x) e_k(t) \eqqcolon \ft^{-1}(\hat u)
\intertext{with}
	\hat u_k(x) &\coloneqq \ft_k(u) \coloneqq \int_{\T} u(x, t) \overline{e_k(t)} \der t.
\end{align*}
Writing $L \coloneqq V(x) \partial_t^2 - \partial_x^2$ and $L_k \coloneqq - \partial_x^2 - k^2 \omega^2 V(x)$, we see that any solution $u$ of \eqref{eq:wave} satisfies
\begin{align*}
	0 = L u
\end{align*}
and therefore also
\begin{align}\label{eq:linear_decomposed}
	0 = \ft_k L u = L_k \ft_k u = L_k \hat u_k
\end{align}
for all $k\in \Z$. Since
\begin{align*}
	\norm{u}_{L^2([0, \infty) \times \T)}^2 + \norm{u_x}_{L^2([0, \infty) \times \T)}^2 
	= \sum_{k \in \Z} \norm{\hat u_k}_{L^2(0, \infty)}^2 + \norm{(\hat u_k)_x}_{L^2(0, \infty)}^2,
\end{align*}
each $\hat u_k$ is an $H^1((0, \infty), \C)$-solution of \eqref{eq:linear_decomposed}. 
As $V$ (and therefore also $L_k$) is given explicitly, we can characterize the space of solutions of \eqref{eq:linear_decomposed} as follows. 

\begin{proposition}\label{prop:fundamental_mode}
	If $k \in \Z$ is even, then the only solution $\hat u_k \in H^1((0, \infty), \C)$ to \eqref{eq:linear_decomposed} is $\hat u_k = 0$. If $k$ is odd, there exists a fundamental Bloch mode $\phi_k \in H^2((0,\infty), \R)$ such that a function $\hat u_k \in H^1((0, \infty), \C)$ solves \eqref{eq:linear_decomposed} if and only if $\hat u_k = \lambda \phi_k$ for some $\lambda \in \C$.
	Furthermore, $\phi_k$ satisfies
	\begin{align*}
		\phi_k(0) = 1, \quad
		\phi_k'(0) = C k (-1)^{(k-1)/2}, \quad
		\phi_k(x + 4 \pi) = \frac{a}{b} \phi_k(x)
	\end{align*}
	for $x > 0$, where  $C = C(T, a) \in \R$ is a constant independent of $k$.
\end{proposition}

A proof of \cref{prop:fundamental_mode} for $k$ odd can be found in \cite[Appendix A2]{kohler_reichel}. The nonexistence result for even $k$ can be obtained using similar arguments:
For $k \neq 0$ the monodromy matrix for $L_k$ is the identity matrix so that \eqref{eq:linear_decomposed} only has spatially periodic solutions. For $k=0$, the solutions of \eqref{eq:linear_decomposed} are affine.

\subsection{Bootstrapping argument}

Assume that $u$ is a weak solution to \eqref{eq:wave} in the sense of Definition~\ref{def:weak_periodic_solution}. By \cref{prop:fundamental_mode}, all even Fourier modes of $u$ vanish so that there exists a complex sequence $\hat \alpha_k$ such that
\begin{align}\label{eq:solution_fourier_series}
	u(x, t) = \sum_{k \in \Zodd} \hat \alpha_k \phi_k(x) e_k(t).
\end{align}
where $\Zodd \coloneqq 2 \Z + 1$. In particular, $u$ is $\frac{T}{2}$-antiperiodic.
Choosing $x = 0$ in \eqref{eq:solution_fourier_series}, we find $u(0, t) = \sum_{k \in \Zodd} \hat \alpha_k e_k(t) \eqqcolon \alpha(t)$.
As $\beta \coloneqq f(u_t(0, \impvar)) \in L^1(\T)$, we can define its Fourier coefficients $\hat \beta_k \coloneqq \ft_k(\beta)$.
The functions $\alpha$ and $\beta$ are related in two ways, which we will exploit to construct a bootstrapping argument. 

Firstly, we have
\begin{align*}
	\alpha'(t) = u_t(0, t) = f^{-1}(f(u_t(0, t))) = f^{-1}(\beta(t)).
\end{align*}
We can apply $\partial_t^{-1}$ to both sides and obtain 
\begin{align}\label{eq:regularity_beta_to_alpha}
	\alpha = \partial_t^{-1} f^{-1}(\beta)
\end{align}
Here $\partial_t^{-1} g \coloneqq \ft^{-1}\bigl((\frac{1}{\ii k \omega}\hat g_k)_{k\in \Zodd}\bigr)$ for a $\frac{T}{2}$-antiperiodic function $g \in L^1(\T)$. Secondly, by using \cref{def:weak_periodic_solution} with $\varphi(x, t) = \psi(x) \overline{e_k(t)}$ for $k \in \Zodd$, where $\psi \in C_c^\infty([0, \infty))$ and $\psi(0) = 1$, we obtain 
\begin{align*}
	0&= \int_{[0, \infty) \times \T} \left[-V(x) u_t \psi(x) \overline{e_k'(t)} + u_x \psi'(x) \overline{e_k(t)}\right] \der (x, t) - \int_\T f(u_t(0, t)) \psi(0) \overline{e_k'(t)} \der t \\
	&\quad= \int_{0}^{\infty} \left[- V(x) \ii k \omega \hat\alpha_k \phi_k(x) \overline{\ii k \omega} \psi(x) + \hat\alpha_k \phi_k'(x) \psi'(x)\right] \der x + \ii k \omega \hat\beta_k \\
	&\quad= \int_{0}^{\infty} \left[- \hat\alpha_k k^2 \omega^2 V(x) \phi_k(x) \psi(x) - \hat\alpha_k \phi_k''(x) \psi(x)\right] \der x - \hat\alpha_k \phi_k'(0) \psi(0) + \ii k \omega \hat\beta_k \\
	&\quad= - \phi_k'(0) \hat\alpha_k + \ii k \omega \hat\beta_k,
\end{align*}
or
\begin{align} \label{eq:regularity_alpha_to_beta}
	\hat \beta_k = \frac{\phi_k'(0)}{\ii k \omega} \hat \alpha_k.
\end{align}
Since $u(0,\cdot)$ is $\frac{T}{2}$-antiperiodic, the even Fourier coefficients of $\alpha=u(0,\cdot)$ vanish, and since $f$ is odd the even Fourier coefficients of $\beta=f(u_t(0,\cdot))$ also vanish.

We next investigate the properties of the maps defined by \eqref{eq:regularity_beta_to_alpha} and \eqref{eq:regularity_alpha_to_beta}, which we consider as maps between the fractional Sobolev-Slobodeckij spaces $W^{s, p}(\T)$. The definition and all employed properties of the spaces $W^{s, p}(\T)$ can be found in \cref{sec:sobolevSpaces}. In the following we use the suffix ``anti'' to denote that the space consists of functions which are $\frac{T}{2}$-antiperiodic in time.

\begin{lemma}\label{lem:regularity_beta_to_alpha}
	The map 
	\begin{align*}
		\beta \mapsto \partial_t^{-1} f^{-1}(\beta)
	\end{align*}
	is well-defined from $W^{s, p}_\mathrm{anti}(\T)$ to $W^{1 + rs, p/r}_\mathrm{anti}(\T)$ for any $s \in [0, 1)$ and $p \in [1, \infty)$ as well as from $C^{0,s}_\mathrm{anti}(\T)$ to $C^{1,r s}_\mathrm{anti}(\T)$ for any $s \in [0, 1]$.
\end{lemma}
\begin{proof}
	If $\beta \in C^{0,s}_\mathrm{anti}(\T)$, then $f^{-1}(\beta) \in C^{0, r s}_\mathrm{anti}(\T)$ since $f^{-1}$ is $r$-Hölder regular, and thus $\partial_t^{-1} f(\beta) \in C^{1, r s}_\mathrm{anti}(\T)$.
	If $\beta \in W^{s, p}_\mathrm{anti}(\T)$, then $f^{-1}(\beta) \in W^{r s, p/r}_\mathrm{anti}(\T)$ by \cref{lem:sob:root_regularity} and thus $\partial_t^{-1} f(\beta) \in W^{1 + r s, p/r}_\mathrm{anti}(\T)$.
\end{proof}

\begin{lemma}\label{lem:regularity_alpha_to_beta}
	The map
	\begin{align*}
		\alpha \mapsto \ft^{-1}\left(\left(\frac{\phi_k'(0)}{\ii k \omega} \hat \alpha_k\right)_{k\in\Zodd}\right)
	\end{align*}
	is well-defined from $W^{s, p}_\mathrm{anti}(\T)$ to $W^{s, p}_\mathrm{anti}(\T)$ for all $s \in (0, \infty)$ and $p \in [1, \infty)$ as well as from $C^{k, s}_\mathrm{anti}(\T)$ to $C^{k, s}_\mathrm{anti}(\T)$ for all $k \in \N_0$ and $s \in [0, 1]$.
\end{lemma}
\begin{proof}
	We begin by taking a closer look at the Fourier multiplier $\hat M_k \coloneqq \frac{\phi_k'(0)}{\ii k \omega}$ which is defined for $k \in \Zodd$ and extended by $0$ to the whole of $\Z$. By \cref{prop:fundamental_mode} we have $\phi_k'(0) = C k (-1)^{(k - 1) / 2}$ for a real constant $C$ depending only on $T$ and $a$. From this we obtain
	\begin{align*}
		\hat M_k = - \frac{\ii C}{\omega} \Im \ii^k
	\end{align*}
	for all $k \in \Z$. Now, $\hat M_k$ is the Fourier series of
	\begin{align*}
		M(t) \coloneqq \frac{\sqrt{T} C }{2 \omega} \left( \delta_{T/4}(t) - \delta_{-T/4}(t) \right)
	\end{align*}
	where $\delta_{x}$ denotes the Dirac measure at $x$. In particular, $M$ is a finite measure. For $\alpha \in L^1_\mathrm{anti}(\T)$ we calculate
	\begin{align*}
		\ft_k \left(\frac{1}{\sqrt{T}} M \ast \alpha\right)
		&= \frac{1}{\sqrt{T}} \int_\T \int_\T  \alpha(t - s) \der M(s) \overline{e_k(t)} \der t \\
		&= \int_\T \int_\T \alpha(t - s) \overline{e_k(t - s)} \der t \overline{e_k(s)} \der M(s)
		= \hat M_k \hat \alpha_k.
	\end{align*}
	so that $\ft^{-1}\left(k \mapsto \hat M_k \hat \alpha_k \right)$ exists and equals $\frac{1}{\sqrt{T}} M \ast \alpha$. To see that $\frac{1}{\sqrt{T}} M \ast (\impvar)$ maps $W^{s,p}_\mathrm{anti}(\T)$ into $W^{s, p}_\mathrm{anti}(\T)$ and $C^{k,s}_\mathrm{anti}(\T)$ into $C^{k, s}_\mathrm{anti}(\T)$, let $\norm{\impvar}$ be $\norm{\impvar}_{W^{s, p}}$ or $\norm{\impvar}_{C^{k,s}}$ (or any translation invariant norm). Then
	\begin{align*}
		\norm{\ft^{-1}\left((\hat M_k\hat\alpha_k)_{k\in\Zodd}\right)} &= \norm{\frac{1}{\sqrt{T}} M \ast \alpha}
		= \frac{1}{\sqrt{T}} \norm{\int_\T \alpha(\impvar - s) \der M(s)} \\
		& \leq \frac{1}{\sqrt{T}} \int_\T \norm{\alpha(\impvar - s)} \der \abs{M}(s)
		= \frac{\abs{M}(\T)}{\sqrt{T}} \norm{\alpha}. &\qedhere
	\end{align*}
\end{proof}

With the previous two lemmata, we can complete the bootstrapping argument stated next.
\begin{lemma}\label{lem:bootstrapping}
	If the pair $(\alpha, \beta)$ satisfies \eqref{eq:regularity_beta_to_alpha} and \eqref{eq:regularity_alpha_to_beta} with $\alpha, \beta \in L^1_\mathrm{anti}(\T)$, then $\alpha, \beta \in C^{1, r}_\mathrm{anti}(\T)$.
\end{lemma}
\begin{proof}
	By \cref{lem:regularity_beta_to_alpha} we have $\alpha \in W^{1, 1 / r}_\mathrm{anti}(\T)$, and therefore $\beta \in W^{1, 1/r}_\mathrm{anti}(\T)$ by \cref{lem:regularity_alpha_to_beta}. 
	Applying \cref{lem:regularity_beta_to_alpha,lem:regularity_alpha_to_beta} again, we get $\alpha, \beta \in W^{1 + r - \varepsilon, 1/r^2}_\mathrm{anti}(\T)$ for any $\varepsilon > 0$. 
	Repeating this $n$ times, we obtain $\alpha, \beta \in W^{1 + r - \varepsilon, 1/r^{2 + n}}_\mathrm{anti}(\T)$. 
	If $n \in \N$ is large enough, then $W^{1+ r - \varepsilon, 1/r^{2 + n}}_\mathrm{anti}(\T)$ embeds continuously into $C^1_\mathrm{anti}(\T)$ by \cref{lem:sob:embedding}, so in particular we have $\alpha, \beta \in C^1_\mathrm{anti}(\T)$. 
	Now, applying \cref{lem:regularity_beta_to_alpha,lem:regularity_alpha_to_beta} one last time yields $\alpha, \beta \in C^{1,r}_\mathrm{anti}(\T)$.
\end{proof}

\begin{proof}[Proof of \cref{thm:regularity}]
	Note that $\alpha, \beta \in L^1_\mathrm{anti}(\T)$ by \cref{def:weak_periodic_solution}, so \cref{lem:bootstrapping} is applicable and yields $\alpha, \beta \in C^{1, r}_\mathrm{anti}(\T)$.

	By $d_1 \coloneqq \theta \pi, d_2 \coloneqq (2 - \theta) \pi, d_3 \coloneqq (2 + \theta) \pi, \dots$ we label the discontinuities of $V$. We start by showing that $u \in C^{1, r}_\mathrm{anti}([0, d_1] \times \T)$. To do this, consider
	\begin{align}\label{eq:loc:solution_formula:1}
		w(x, t) \coloneqq \frac{1}{2} \left( \alpha(t + \sqrt{a} x) + \alpha(t - \sqrt{a} x) \right) + \frac{1}{2 \sqrt{a}} \left( \beta(t + \sqrt{a} x) - \beta(t - \sqrt{a} x) \right).
	\end{align}
	Note that $w$ is $\frac{T}{2}$-antiperiodic in time. The $k$-th Fourier coefficient of $w$ is given by
	\begin{align*}
		\hat w_k(x) 
		&= \frac{\hat \alpha_k}{2} \left( \ee^{i k \omega \sqrt{a} x} + \ee^{- \ii k \omega \sqrt{a} x} \right) + \frac{\hat \beta_k}{2 \sqrt{a}} \left( \ee^{i k \omega \sqrt{a} x} - \ee^{- \ii k \omega \sqrt{a} x} \right) \\
		&= \hat \alpha_k \cos(k \omega \sqrt{a} x) + \frac{\hat \beta_k \ii}{\sqrt{a}} \sin(k \omega \sqrt{a} x).
	\end{align*}
	We see that $\hat w_k$ solves $L_k \hat w_k = 0$ on $[0, d_1]$ and at $x = 0$ it satisfies
	\begin{align*}
		\hat w_k(0) = \hat \alpha_k = \hat \alpha_k \phi_k(0)
		\quad\text{and}\quad
		\hat w_k'(0) = \frac{\hat \beta_k \ii}{\sqrt{a}} k \omega \sqrt{a} = \hat \alpha_k \phi_k'(0),
	\end{align*}
	where we have used \eqref{eq:regularity_alpha_to_beta}. So $\hat w_k(x) = \alpha_k \phi_k(x)$ must hold, and from this we obtain
	\begin{align*}
		w(x, t) = \sum_{k \in \Zodd} \hat w_k(x) e_k(t) = \sum_{k \in \Zodd} \hat \alpha_k \phi_k(x) e_k(t) = u(x, t).
	\end{align*}
	As $w$ is given by \eqref{eq:loc:solution_formula:1}, $u = w \in C^{1, r}_\mathrm{anti}([0, d_1] \times \T)$ follows immediately.

	Now assume that $u \in C^{1, r}_\mathrm{anti}([0, d_n]\times\T)$ holds for some $n \in \N$. We aim to show $u \in C^{1, r}_\mathrm{anti}([0, d_{n + 1}])$, 
	denote by $v \in \set{a,b}$ the value of $V$ on $(d_n, d_{n + 1})$ and define a function $w$ by
	\begin{align}\label{eq:loc:solution_formula:2}
		w(x, t) = \frac12 \left( u(d_n, t + \sqrt{v}(x - d_n)) + u(d_n, t - \sqrt{v}(x - d_n)) \right) + \frac{1}{2 \sqrt{v}} \int_{t - \sqrt{v}(x - d_n)}^{t + \sqrt{v}(x - d_n)} u_x(d_n, \tau) \der \tau
	\end{align}
	for $x \in [d_n, d_{n + 1}]$ and $t \in \T$. Then $w \in C^{1, r}_\mathrm{anti}([d_n, d_{n + 1}] \times \T)$ follows immediately from \eqref{eq:loc:solution_formula:2}. 
	Arguing as above, one can show $L_k \hat w_k(x) = 0$ for all $k \in \Z$. Since $\hat w_k(d_n) = \hat u_k(d_n) = \hat \alpha_k \phi_k(d_n)$ and $\hat w_k'(d_n) = \hat \alpha_k \phi_k'(d_n)$, we again get $\hat w_k(x) = \hat \alpha_k \phi_k(x)$ and thus $w = u$ on $[d_n, d_{n + 1}] \times \T$.

	Next we need to show the uniform bound $\abs{u(x, t)} \leq C \ee^{- \rho x}$ with $\rho = \frac{\log(b) - \log(a)} {4 \pi}$.
	By \cref{prop:fundamental_mode}, $u$ satisfies $u(x + 4 \pi, t) = \frac{a}{b} u(x, t)$ for all $x \in [0, \infty)$ and $t \in \T$. Hence we can choose
	\begin{align*}
		C \coloneqq \max_{x \in [0, 4 \pi], t \in \T} \ee^{\rho x} \abs{u(x, t)}.
	\end{align*}
	To show that $u$ is a $C^1$-solution to \eqref{eq:IP}, first from \eqref{eq:loc:solution_formula:1} it follows that the directional derivative
	\begin{align*}
		(\partial_t - c(x) \partial_x) (u_t + c(x) u_x)
	\end{align*}
	exists and equals $0$ for $x \in (0, d_1)$ as $c(x) = \frac{1}{\sqrt{a}}$ here. Similarly, using \eqref{eq:loc:solution_formula:2} we obtain
	\begin{align*}
		(\partial_t - c(x) \partial_x) (u_t + c(x) u_x) = 0
	\end{align*}
	for $x \in (d_n, d_{n+1})$ as $c(x) = \frac{1}{\sqrt{v}}$. 
	Lastly, due to \eqref{eq:solution_fourier_series}, \eqref{eq:regularity_alpha_to_beta} and the definition of $\beta$ we have
	\begin{align*}
		\ft_k(u_x(0, \impvar)) = \phi_k'(0) \hat \alpha_k = i k \omega \hat \beta_k = \ft_k(\beta') = \ft_k(f(u_t(0, \impvar))_t)
	\end{align*}
	for all $k \in \Zodd$, so $u_x(0, t) = (f(u_t(0, t)))_t$ for all $t \in \T$. This shows that $u$ is a $C^1$-solution to \eqref{eq:IP} with its own initial data.
\end{proof}

	\appendix
	\section{} \label{appendix_A}

\begin{lemma}\label{lem:aux:chain_rule_formula} 
	For $t_0, t_1 \in \R$ with $t_0 < t_1$ and $g \in C([t_0, t_1], \R)$ with $f \circ g$ is $C^1([t_0, t_1])$, the equation
	\begin{align*}
		F(g(t_1)) - F(g(t_0)) = \int_{t_0}^{t_1} g(t) \dv{f(g(t))}{t} \der t
	\end{align*}
	holds.
\end{lemma}

\begin{proof} 
Assume first that $f$ and $g$ are both $C^1$ in which case the definition $F(y)=yf(y)-\int_0^y f(s)\,ds$ and integration by parts yield the result
	\begin{align}\label{eq:loc:boundary_under_ass_c1}
		\begin{split}
			\int_{t_0}^{t_1} g(t) \dv{f(g(t))}{t} \der t 
			&= \left[g(t) f(g(t))\right]_{t = t_0}^{t_1} - \int_{t_0}^{t_1} g'(t)f(g(t)) \der t \\
			&= \left[g(t) f(g(t))\right]_{t = t_0}^{t_1} - \int_{g(t_0)}^{g(t_1)} f(v) \der v
			= F(g(t_1)) - F(g(t_0)).
		\end{split}
	\end{align}
	For the general case, choose a sequence of non-negative smooth mollifiers $\phi_n \colon \R \to [0, \infty)$ converging to $\delta_0$, each with support in $[-\tfrac1n, \tfrac1n]$ and with average $\int_\R \phi_n(x) \der x = 1$. Since $f$ is strictly increasing, so is $f_n \coloneqq \phi_n \ast f$. In particular, $f_n$ is bijective and we may define $g_n \coloneqq \left(f_n\right)^{-1} \circ f \circ g$ so that $f_n\circ g_n=f\circ g$. 
	
	Clearly, $f_n \to f$ uniformly on compacts. To see that $g_n \to g$ uniformly on compacts, it suffices to show $\norm{\left(f_n\right)^{-1} - f^{-1}}_\infty \leq \tfrac1n$ for $n \in \N$. Note that
	\begin{align*}
		f_n(x - \tfrac1n) = \int_{x - \tfrac2n}^{x} f(y) \phi_n(x - \tfrac1n - y) \der y \leq \int_{x - \tfrac2n}^{x} f(x) \phi_n(x - \tfrac1n - y) \der y = f(x).
	\end{align*}
	If we choose $x \coloneqq f^{-1}(y)$ for arbitrary $y \in \R$ and apply $\left(f_n\right)^{-1}$ to both sides of the above inequality, we get $f^{-1}(y) - \tfrac1n \leq \left(f_n\right)^{-1}(y)$. Similarly, $f^{-1}(y) + \tfrac1n \geq \left(f_n\right)^{-1}(y)$ holds so that the estimate $\norm{\left(f_n\right)^{-1} - f^{-1}}_\infty \leq \tfrac1n$ is shown. Letting $F_n(s) \coloneqq s f_n(s) - \int_{0}^{s} f_n(\sigma) \der \sigma$, by \eqref{eq:loc:boundary_under_ass_c1} we have
	\begin{align*}
		F_n(g_n(t_1)) - F_n(g_n(t_0)) = \int_{t_0}^{t_1} g_n(t) \dv{f_n(g_n(t))}{t} \der t = \int_{t_0}^{t_1} g_n(t) \dv{f(g(t))}{t} \der t.
	\end{align*}
	For $n \to \infty$, the desired result follows. 
\end{proof}

\section{Sobolev-Slobodeckij space} \label{sec:sobolevSpaces}

\begin{definition}\label{def:sob:spaces}
	Denote the distance on the torus $\T$ by $d$. Then, for $s \in (0, 1)$ and $p \in [1, \infty)$ define the Sobolev-Slobodeckij space $W^{s, p}(\T) \coloneqq \set{u \in L^p(\T) \colon \seminorm{u}_{W^{s, p}(\T)} < \infty}$ with 
	\begin{align*}
		\seminorm{u}_{W^{s, p}(\T)}^p = \int_\T \int_\T \frac{\abs{u(t_1) - u(t_2)}^p}{d(t_1, t_2)^{1 + s p}} \der t_1 \der t_2
	\end{align*}
	Also let $W^{0, p}(\T) \coloneqq L^p(\T)$ and $W^{k + s, p}(\T) \coloneqq \set{u \in W^{k, p}(\T) \colon u^{(k)} \in W^{s, p}(\T)}$ for $k \in \N$, $s \in [0, 1)$ and $p \in [1, \infty)$.
\end{definition}

\begin{lemma}\label{lem:sob:root_regularity}
	If $g \colon \R \to \R$ is $r$-Hölder continuous, then the map 
	\begin{align*}
		W^{s, p}(\T) \to  W^{r s, p / r}(\T), u \mapsto g \circ u
	\end{align*}
	is well-defined for $s \in [0, 1)$ and $p \in [1, \infty)$.
\end{lemma}

\begin{proof}
	By assumption, there exists $C > 0$ such that $\abs{g(x) - g(y)} \leq C \abs{x - y}^{r}$ holds for all $x, y \in \R$.
	First, let $u \in L^p(\T)$.  Then
	\begin{align*}
		\norm{g(u)}_{L^{p/r}(\T)}^{p/r} 
		&= \int_\T \abs{g(u(t))}^{p/r} \der t
		\leq 2^{p/r - 1} \int_\T \left(\abs{g(u(t)) - g(0)}^{p/r} + \abs{g(0)}^{p/r}\right) \der t \\
		&\leq 2^{p/r - 1} \int_\T \left(C^{p/r} \abs{u(t)}^{p} + \abs{g(0)}^{p/r}\right) \der t 
		= 2^{p/r - 1} \left( C^{p/r} \norm{u}_{L^p(\T)}^p + T \abs{g(0)}^{p/r} \right),
	\end{align*}
	so $g(u) \in L^{p/r}(\T)$. Now let $u \in W^{s, p}(\T)$ with $s \in (0, 1)$. Then
	\begin{align*}
		\seminorm{g(u)}_{W^{r s, p / r}(\T)}^{p / r} 
		&= \int_\T \int_\T \frac{\abs{g(u(t_1)) - g(u(t_2))}^{p/r}}{d(t_1, t_2)^{1 + s p}} \der t_1 \der t_2 \\
		&\leq \int_\T \int_\T \frac{C^{p/r} \abs{u(t_1) - u(t_2)}^{p}}{d(t_1, t_2)^{1 + s p}} \der t_1 \der t_2
		= C^{p/r} \seminorm{u}_{W^{s, p}(\T)}^p. \qedhere
	\end{align*}
\end{proof}

\begin{lemma}\label{lem:sob:embedding}
	$W^{1 + s, p}(\T) \embeds C^{1,s-\frac{1}{p}}(\T)$ for $s \in (0, 1), p \in (1, \infty)$ with $s p > 1$.
\end{lemma}
\begin{proof}
	Consider the fractional Sobolev-Slobodeckij space $W^{s, p}([0, T])$ which is similarly defined using the seminorm
	\begin{align*}
		\seminorm{v}_{W^{s, p}([0, T])}^p = \int_0^T \int_0^T \frac{\abs{v(t_1) - v(t_2)}^p}{\abs{t_1 - t_2}^{1 + s p}} \der t_1 \der t_2
	\end{align*}
	We have $\seminorm{u'}_{W^{s, p}([0, T])}^p \leq \seminorm{u'}_{W^{s, p}(\T)}^p < \infty$, so that  $u' \in W^{s, p}([0, T])$ and from \cite[Theorem 2]{Hitch} it follows that $u' \in C^{(s p - 1) / p}([0, T])$.
\end{proof}

	\section*{Acknowledgment}
	Funded by the Deutsche Forschungsgemeinschaft (DFG, German Research Foundation) – Project-ID 258734477 – SFB 1173. 

    \printbibliography
\end{document}